\documentclass{amsart}
\usepackage{amsthm,amsmath,amssymb,amscd, mathrsfs}
\usepackage[latin1]{inputenc}
\usepackage[all]{xy}
\usepackage{url}
\DeclareMathOperator{\pr}{pr}
\DeclareMathOperator{\Ker}{Ker}

\DeclareMathOperator{\Image}{Im}

\DeclareMathOperator{\Hom}{Hom}
\DeclareMathOperator{\Gal}{Gal}
\DeclareMathOperator{\Frob}{Frob}
\DeclareMathOperator{\id}{id}

\DeclareMathOperator{\tr}{tr}

\DeclareMathOperator{\Spec}{Spec}
\DeclareMathOperator{\Res}{Res}

\DeclareMathOperator{\Br}{Br}

\DeclareMathOperator{\GL}{GL}

\DeclareMathOperator{\rank}{rank}

\DeclareMathOperator{\tame}{tame}
\DeclareMathOperator{\LS}{LS}

\newcommand{\ra}{\rightarrow}
\newcommand{\lra}{\longrightarrow}

\newcommand{\hra}{\hookrightarrow}

\newcommand{\N}{\mathbb{N}}
\newcommand{\Z}{\mathbb{Z}}
\newcommand{\Q}{\mathbb{Q}}
\newcommand{\R}{\mathbb{R}}
\newcommand{\C}{\mathbb{C}}
\newcommand{\A}{\mathbb{A}}
\newcommand{\F}{\mathbb{F}}

\newcommand{\calO}{\mathcal{O}}
\newcommand{\calE}{\mathcal{E}}
\newcommand{\calF}{\mathcal{F}}
\newcommand{\calG}{\mathcal{G}}

\newcommand{\fkm}{\mathfrak{m}}

\newcommand{\fkT}{\mathfrak{T}}

\theoremstyle{plain}
\newtheorem{thm}{Theorem}[section]
\newtheorem*{thm*}{Theorem}
\newtheorem{prop}[thm]{Proposition}
\newtheorem{lem}[thm]{Lemma}
\newtheorem{cor}[thm]{Corollary}
\newtheorem{conj}[thm]{Conjecture}

\theoremstyle{definition}
\newtheorem{defn}[thm]{Definition}

\theoremstyle{remark}
\newtheorem{rem}[thm]{Remark}
\newtheorem{claim}{Claim}
\newtheorem*{claim*}{Claim}

\begin{document}
\title{
Existence of compatible systems of lisse sheaves on arithmetic schemes
}
\date{}
\author{Koji Shimizu}
\address{Harvard University, Department of Mathematics, Cambridge, MA 02138}
\email{shimizu@math.harvard.edu}

\begin{abstract}
Deligne conjectured that a single $\ell$-adic lisse sheaf on a normal variety over a finite field can be embedded into a compatible system of $\ell'$-adic lisse sheaves with various $\ell'$. 
Drinfeld used Lafforgue's result as an input and proved this conjecture when the variety is smooth.
We consider an analogous existence problem for a regular flat scheme over $\Z$
and prove some cases using Lafforgue's result and the work of Barnet-Lamb, Gee, Geraghty, and Taylor.
\end{abstract}

\maketitle

\section{Introduction}

In \cite{MR601520}, Deligne 
conjectured that all the $\overline{\Q}_\ell$-sheaves on a variety over a finite field are mixed. A standard argument reduces this conjecture to the following one.

\begin{conj}[Deligne]\label{conj:Deligne}
Let $p$ and $\ell$ be distinct primes.
 Let $X$ be a connected normal scheme of finite type over $\F_p$ and $\calE$ an irreducible lisse $\overline{\Q}_\ell$-sheaf whose determinant has finite order.
Then the following properties hold:
\begin{enumerate}
 \item $\calE$ is pure of weight 0.
 \item There exists a number field $E\subset \overline{\Q}_\ell$ such that the polynomial $\det(1-\Frob_xt,\calE_{\bar{x}})$ has coefficients in $E$ for every $x\in \lvert X\rvert$.
 \item For every non-archimedean place $\lambda$ of $E$ prime to $p$, the 
roots of $\det(1-\Frob_xt,\calE_{\bar{x}})$ are $\lambda$-adic units.
 \item For a sufficiently large $E$ and for every non-archimedean place $\lambda$ of $E$ prime to $p$, there exists an $E_\lambda$-sheaf $\calE_\lambda$ compatible with $\calE$, that is, 
$\det (1-\Frob_xt,\calE_{\bar{x}})=\det (1-\Frob_xt,\calE_{\lambda,\bar{x}})$
for every $x\in \lvert X\rvert$.
\end{enumerate}
Here $\lvert X\rvert$ denotes the set of closed points of $X$ and
$\bar{x}$ is a geometric point above $x$.
\end{conj}

The conjecture for curves is proved by L. Lafforgue in \cite{MR1875184}.
He also deals with parts (i) and (iii) in general by reducing them to the case of curves (see also \cite{MR3024820}).
Deligne proves part (ii) in \cite{MR3024820}, and Drinfeld proves part (iv) for smooth varieties in \cite{MR3024821}. They both use Lafforgue's results.

We can consider similar questions for arbitrary schemes of finite type over $\Z[\ell^{-1}]$.
This paper focuses on part (iv), namely the problem of embedding a single lisse $\overline{\Q}_\ell$-sheaf into a compatible system of lisse sheaves. We have the following folklore conjecture in this direction.

\begin{conj}\label{conj:compatibility w/o finiteness}
Let $\ell$ be a rational prime.
Let $X$ be an irreducible regular scheme that is flat and of finite type over
$\Z[\ell^{-1}]$.
Let $E$ be a finite extension of $\Q$ and $\lambda$ a prime of $E$ above $\ell$.
Let $\calE$ be an irreducible lisse $E_\lambda$-sheaf on $X$ and $\rho$ the corresponding representation of $\pi_1(X)$.
Assume the following conditions:
\begin{enumerate}
 \item The polynomial $\det (1-\Frob_xt,\calE_{\bar{x}})$ has coefficients in $E$ for every $x\in\lvert X\rvert$.
 \item $\calE$ is de Rham at $\ell$ (see below for the definition).
\end{enumerate}
Then for each rational prime $\ell'$ and each prime $\lambda'$ of $E$ above $\ell'$
 there exists a lisse $\overline{E}_{\lambda'}$-sheaf on $X[\ell'^{-1}]$ which is compatible with $\calE|_{X[\ell'^{-1}]}$.
\end{conj}

The conjecture when $\dim X=1$ is usually rephrased in terms of Galois representations of a number field (see Conjecture~1.3 of \cite{MR1989198} for example).

When $\dim X>1$,  a lisse sheaf $\calE$ is called \textit{de Rham at $\ell$}
if for every closed point $y\in X\otimes \Q_\ell$, the representation $i_y^\ast\rho$
of $\Gal\bigl(\overline{k(y)}/k(y)\bigr)$ is de Rham, where $i_y$ is the morphism $\Spec k(y)\ra X$. 
Ruochuan Liu and Xinwen Zhu have shown
that this is equivalent to the condition that
the lisse $E_\lambda$-sheaf $\calE|_{X\otimes \Q_\ell}$
on $X\otimes \Q_\ell$ is a de Rham sheaf
in the sense of relative $p$-adic Hodge theory 
(see \cite{LZ} for details).

Now we discuss our main results.
They concern Conjecture~\ref{conj:compatibility w/o finiteness}
for schemes over the ring of integers of a totally real or CM field.

\begin{thm}\label{thm:main thm-special}
Let $\ell$ be a rational prime and $K$ a totally real field which is unramified at $\ell$.
Let $X$ be an irreducible smooth $\calO_K[\ell^{-1}]$-scheme such that 
\begin{itemize}
 \item the generic fiber is geometrically irreducible,
 \item $X_K(\R)\neq \emptyset$ for every real place $K\hra \R$, and
 \item $X$ extends to an irreducible smooth $\calO_K$-scheme with nonempty fiber over each place of $K$ above $\ell$.
\end{itemize}
Let $E$ be a finite extension of $\Q$ and $\lambda$ a prime of $E$ above $\ell$.
Let $\calE$ be a lisse $E_\lambda$-sheaf on $X$ and $\rho$ the corresponding representation of $\pi_1(X)$.
Suppose that $\calE$ satisfies the following assumptions:
\begin{enumerate}
 \item The polynomial $\det (1-\Frob_xt,\calE_{\bar{x}})$ has coefficients in $E$ for every $x\in \lvert X\rvert$.
 \item For every totally real field $L$ which is unramified at $\ell$ and every morphism $\alpha\colon \Spec L\ra X$, 
the $E_\lambda$-representation $\alpha^\ast\rho$ of $\Gal(\overline{L}/L)$ is crystalline at each prime $v$ of $L$ above $\ell$, and for each
$\tau\colon L\hra \overline{E}_\lambda$ it has 
distinct $\tau$-Hodge-Tate numbers in the range $[0,\ell-2]$.
 \item $\rho$ can be equipped with symplectic (resp.~ orthogonal) structure with multiplier $\mu\colon \pi_1(X)\ra E_\lambda^\times$ 
such that $\mu|_{\pi_1(X_K)}$ admits a factorization
\[
\mu|_{\pi_1(X_K)}\colon \pi_1(X_K)\lra\Gal(\overline{K}/K)\stackrel{\mu_K}{\lra} E_\lambda^\times
\]
with a totally odd (resp.~ totally even) character $\mu_K$
(see below for the definitions).
 \item The residual representation $\overline{\rho}|_{\pi_1(X[\zeta_\ell])}$ is absolutely irreducible.
Here $\zeta_\ell$ is a primitive $\ell$-th root of unity and
$X[\zeta_\ell]=X\otimes_{\calO_K}\calO_K[\zeta_\ell]$.
 \item $\ell\geq 2(\rank \calE+1)$.
\end{enumerate}
Then for each rational prime $\ell'$ and each prime $\lambda'$ of $E$ above $\ell'$ there exists a lisse $\overline{E}_{\lambda'}$-sheaf on $X[\ell'^{-1}]$ which is compatible with $\calE|_{X[\ell'^{-1}]}$.
\end{thm}

For an $E_\lambda$-representation $\rho\colon\pi_1(X)\ra \GL(V_\rho)$,
a symplectic (resp.~ orthogonal) structure with multiplier 
is a pair $(\langle\,,\rangle,\mu)$ consisting of a symplectic (resp.~ orthogonal) pairing $\langle\,,\rangle\colon V_\rho\times V_\rho\ra E_\lambda$
and a continuous homomorphism $\mu\colon \pi_1(X)\ra E_\lambda^\times$
satisfying 
$\langle \rho(g)v,\rho(g)v'\rangle=\mu(g)\langle v,v'\rangle$
for any $g\in \pi_1(X)$ and $v,v'\in V_\rho$.

We show a similar theorem without assuming that $K$ is unramified at $\ell$ using the potential diagonalizability assumption. 
See Theorem~\ref{thm:main thm} for this statement and 
Theorem~\ref{thm:main thm for CM} for the corresponding statement when $K$ is CM.

The proof of Theorem~\ref{thm:main thm-special} uses Lafforgue's work and the work of
Barnet-Lamb, Gee, Geraghty, and Taylor (Theorem C of \cite{MR3152941}).
The latter work concerns
Galois representations of  a totally real field, and
it can be regarded as a special case of Conjecture~\ref{conj:compatibility w/o finiteness} when $\dim X=1$.
We remark that their theorem has several assumptions on Galois representations since they use potential automorphy.
Hence Theorem~\ref{thm:main thm-special} needs assumptions (ii)-(v) on lisse sheaves.

The main part of this paper is devoted to constructing a compatible system of 
lisse sheaves on a scheme from those on curves.
Our method is modeled after Drinfeld's result in \cite{MR3024821}, 
which we explain now.

For a given lisse sheaf on a scheme, one can obtain a lisse sheaf on 
each curve on the scheme by restriction. 
Conversely, Drinfeld proves that a collection of lisse sheaves on curves 
on a regular scheme defines a lisse sheaf on the scheme
if it satisfies some compatibility and tameness conditions
(Theorem 2.5 of \cite{MR3024821}). See also a remark after Theorem~\ref{thm:Dr-thm2.5}. 
This method originates from the work of Wiesend on 
higher dimensional class field theory (\cite{MR2268759}, \cite{MR2541032}). 

Drinfeld uses this method to
 reduce part (iv) of Conjecture~\ref{conj:Deligne} for smooth varieties to the case when $\dim X=1$, where he can use Lafforgue's result.
Similarly, one can use his result to reduce Conjecture~\ref{conj:compatibility w/o finiteness} to the case when $\dim X=1$.

However, Drinfeld's result cannot be used to reduce 
Theorem~\ref{thm:main thm-special} to the results of Lafforgue and 
Barnet-Lamb, Gee, Geraghty, and Taylor since
his theorem needs
a lisse sheaf on every curve on the scheme as an input.
On the other hand, 
the results of \cite{MR1875184} and \cite{MR3152941} only provide
compatible systems of lisse sheaves 
on special types of curves on the scheme: 
curves over finite fields and 
\emph{totally real curves}, that is, 
 open subschemes of the spectrum of the ring of integers of a totally real field.
Thus the goal of this paper 
is to deduce  Theorem~\ref{thm:main thm-special}
using the existence of compatible systems of lisse sheaves
on these types of curves.

We now explain our  method.
Fix a prime $\ell$ and a finite extension $E_\lambda$ of $\Q_\ell$.
Fix a positive integer $r$.
On a normal scheme $X$ of finite type over $\Spec \Z[\ell^{-1}]$, each lisse $E_\lambda$-sheaf $\calE$ of rank $r$ defines
a polynomial-valued map $f_{\calE}\colon \lvert X\rvert \ra E_\lambda[t]$
of degree $r$
given by $f_{\calE, x}(t)=\det (1-\Frob_xt,\calE_{\bar{x}})$.
Here we say that a polynomial-valued map is of degree $r$ if
its values are polynomials of degree $r$.
Moreover, $f_\calE$ determines $\calE$ up to semisimplifications by the Chebotarev density theorem.
Conversely, we can ask whether a polynomial-valued map $f\colon\lvert X\rvert\ra E_\lambda[t]$ of degree $r$ arises from a lisse sheaf of rank $r$ on $X$ in this way.

Let $K$ be a totally real field. 
Let $X$ be an irreducible smooth $\calO_K$-scheme which has geometrically irreducible generic fiber and satisfies $X_K(\R)\neq \emptyset$ for every real place $K\hra \R$.
In this situation, we show the following theorem.

\begin{thm}\label{thm:Dr-thm2.5}
A polynomial-valued map $f$ of degree $r$ on $\lvert X\rvert$ arises from a lisse sheaf on $X$
if and only if it satisfies the following conditions:
 \begin{enumerate}
  \item The restriction of $f$ to each totally real curve arises from a lisse sheaf.
  \item The restriction of $f$ to each separated smooth curve over a finite field arises from a lisse sheaf.
 \end{enumerate}
\end{thm}

We prove a similar theorem when $K$ is CM (Theorem~\ref{thm:Dr-thm2.5 for CM}).

Drinfeld's theorem involves a similar equivalence, which holds for arbitrary regular schemes of finite type, although his condition (i) is required to hold for arbitrary regular curves and there is an additional tameness assumption in his condition (ii).\footnote{We do not need tameness assumption in condition (ii) in Theorem~\ref{thm:Dr-thm2.5}. This was pointed out by Drinfeld.}

If $K$ and $X$ satisfy the assumptions in Theorem~\ref{thm:main thm-special}, 
then we prove a variant of Theorem~\ref{thm:Dr-thm2.5}, 
where we require condition (i) to hold only for totally real curves which are unramified over $\ell$ (Remark~\ref{rem:var of Dr-thm2.5 for unramified case}).
This variant, combined with the results by Lafforgue and Barnet-Lamb, Gee, Geraghty, and Taylor, implies Theorem~\ref{thm:main thm-special}.

One of the main ingredients for the proof of these types of theorems is an approximation theorem: One needs to find a curve passing through given points in given tangent directions and satisfying a technical condition
coming from a given \'etale covering.
To prove this Drinfeld uses the Hilbert irreducibility theorem. 
In our case, we need to further require that such a curve  be totally real or CM. For this we use a theorem of Moret-Bailly.

We briefly mention a topic related to Conjecture~\ref{conj:compatibility w/o finiteness}.
As conjectures on Galois representations suggest, the following stronger statement should hold.

\begin{conj}\label{conj:compatibility}
With the notation as in Conjecture~\ref{conj:compatibility w/o finiteness},
condition \emph{(ii)} implies condition \emph{(i)}
after replacing $E$ by a bigger number field inside $E_\lambda$.
\end{conj}

This is an analogue of Conjecture~\ref{conj:Deligne} (ii) and (iii).
Even if we assume the conjectures for curves, that is, Galois representations of a number field, no method is known to prove Conjecture~\ref{conj:compatibility} in full generality. 
However in \cite{Shi} 
we show the conjecture for smooth schemes assuming conjectures on Galois representations of a number field and the Generalized Riemann Hypothesis. 
Note that  Deligne's proof of Conjecture~\ref{conj:Deligne} (iii) (\cite{MR3024820}) uses the Riemann Hypothesis for varieties over finite fields, or more precisely, the purity theorem of \cite{MR601520}.

We now explain the organization of this paper.
In Section~\ref{section:tatally real curves}, we review the theorem of Moret-Bailly and prove an approximation theorem for ``schemes with enough totally real curves.''  We show a similar theorem in  the CM case.
In Section~\ref{section:prf of Dr-thm2.5}, we prove Theorem~\ref{thm:Dr-thm2.5} and its variants using the approximation theorems. Most arguments in Section~\ref{section:prf of Dr-thm2.5} originate from Drinfeld's paper \cite{MR3024821}. 
Finally, we prove main theorems in Section~\ref{section:prf of main thm}.

\textbf{Notation.}
For a number field $E$ and a place $\lambda$ of $E$, we denote by $\overline{E}_\lambda$ a fixed algebraic closure of $E_\lambda$.

For a scheme $X$, we denote by $\lvert X\rvert$ the set of closed points of $X$. We equip finite subsets of $\lvert X\rvert$ with the reduced scheme structure. 
We denote the residue field of a point $x$ of $X$ by $k(x)$.
An \'etale covering over $X$ means a scheme which is finite and \'etale over $X$.

For a number field $K$ and an $\calO_K$-scheme $X$, $X_K$ denotes the generic fiber of $X$ regarded as a $K$-scheme.
In particular, for a $K$-algebra $R$, $X_K(R)$ means $\Hom_K(\Spec R, X_K)$, 
not $\Hom_{\Z}(\Spec R, X_K)$. We also write $X(R)$ instead of $X_K(R)$.

For simplicity, we omit base points of  fundamental groups and
we often change base points implicitly in the paper.

\textbf{Acknowledgments.}
I would like to thank Takeshi Saito for introducing
me to Drinfeld's paper and Mark Kisin for suggesting this topic to me.
This work owes a significant amount to the work of Drinfeld.
He also gave me important suggestions on the manuscript, which simplified some of the main arguments.
I would like to express my sincere admiration and gratitude to Drinfeld for his work and comments. 
Finally, it is my pleasure to thank George Boxer for a clear explanation 
of potential automorphy and many suggestions on the manuscript,
and Yunqing Tang for a careful reading of the manuscript
and many useful remarks.

\section{Existence of totally real and CM curves via the theorem of Moret-Bailly }
\label{section:tatally real curves}
First we recall the theorem of Moret-Bailly.

\begin{thm}[{Moret-Bailly, \cite[II]{MR1005158}}] 
\label{thm:Moret-Bailly}
Let $K$ be a number field. We consider a quadruple $(X_K, \Sigma, \{M_v\}_{v\in \Sigma}, \{\Omega_v\}_{v\in \Sigma})$ consisting of
\begin{enumerate}
 \item a geometrically irreducible, smooth and separated $K$-scheme $X_K$,
 \item a finite set $\Sigma$ of places of $K$,
 \item a finite Galois extension $M_v$ of $K_v$ for every $v\in\Sigma$, and
 \item a nonempty $\Gal(M_v/K_v)$-stable open subset $\Omega_v$ of $X_K(M_v)$ with respect to $M_v$-topology.
\end{enumerate}
Then there exist a finite extension $L$ of $K$ and an $L$-rational point $x\in X_K(L)$ satisfying the following two conditions:
\begin{itemize}
 \item For every $v\in \Sigma$, $L$ is $M_v$-split, that is, $L\otimes_KM_v\cong M_v^{[L\colon K]}$.
 \item The images of $x$ in $X_K(M_v)$ induced from embeddings $L\hra M_v$ lie in $\Omega_v$.
\end{itemize}
\end{thm}

\begin{rem} 
 Our formulation is slightly different from Moret-Bailly's, but 
Theorem~\ref{thm:Moret-Bailly}
is a simple consequence of Th\'eor\`eme~1.3 of \cite[II]{MR1005158}.
Namely, we can always find an integral model $f\colon X\ra B$
of $X_K\ra \Spec K$ over a sufficiently small open subscheme $B$ of $\Spec \calO_K$
such that 
$(X\ra B, \Sigma, \{M_v\}_{v\in \Sigma}, \{\Omega_v\}_{v\in \Sigma})$
is an incomplete Skolem datum (see D\'efinition~1.2 of \cite[II]{MR1005158}).
Then Theorem~\ref{thm:Moret-Bailly} follows from Th\'eor\`eme~1.3 of \cite[II]{MR1005158} applied to this incomplete Skolem datum.
\end{rem}

Since the set $\Sigma$ can contain infinite places,
the above theorem implies the existence of totally real or CM valued points.

\begin{lem}\label{lem:totally real points}\hfill
\begin{enumerate}
 \item Let $K$ be a totally real field and $X_K$ a geometrically irreducible smooth $K$-scheme such that $X_K(\R)\neq\emptyset$ for every real place $K\hra \R$. 
For any dense open subscheme $U_K$ of $X_K$, there exists a totally real extension $L$ of $K$ such that $U_K(L)\neq\emptyset$.
 \item Let $F$ be a CM field and $Z_F$ a geometrically irreducible smooth $F$-scheme.
For any dense open subscheme $V_F$ of $Z_F$, there exists a CM extension $L$ of $F$ such that $V_F(L)\neq\emptyset$.
\end{enumerate}
\end{lem}

\begin{proof}
In either setting, we may assume that the scheme is separated over the base field by replacing it by an open dense subscheme.

First we prove (i).
For every real place $v\colon K\hra \R$, let $U_v=U_K\cap (X_K\otimes_{K,v}\R)(\R)$.
It follows from the assumptions and the implicit function theorem that
$U_v$ is a nonempty open subset of $(X_K\otimes_{K,v}\R)(\R)$
with respect to real topology.

We apply the theorem of Moret-Bailly 
 to the datum $\bigl(X_K,\{v\}, \{\R\}_v,\{U_v\}_v\bigr)$ to find a finite extension $L$ of $K$ and a point $x\in X_K(L)$ such that $L\otimes_K\R\cong \R^{[L:K]}$ and the images of $x$ induced from real embeddings $L\hra \R$ above $v$ lie in $U_v$.
Then $L$ is totally real and $x\in U_K(L)$. Hence $U_K(L)\neq \emptyset$.

Next we prove (ii).
Let $F^+$ be the maximal totally real subfield of $F$.
Define
$Z^+_{F^+}$ (resp.~ $V^+_{F^+}$)
to be the Weil restriction $\Res_{F/F^+}Z_F$
(resp.~ $\Res_{F/F^+}V_F$).
Denote the nontrivial element of $\Gal(F/F^+)$ by $c$ 
and $Z_F\otimes_{F,c}F$ by $c^\ast Z_F$.
Then we have
 $Z^+_{F^+}\otimes_{F^+}F\cong Z_F\times_{F}c^\ast Z_F$,
and this scheme is geometrically irreducible over $F$.
Thus $Z^+_{F^+}$ is a geometrically irreducible smooth $F^+$-scheme, and
$V^+_{F^+}$ is dense and open in $Z^+_{F^+}$.
Moreover, for every real place $F^+\hra\R$, we can extend it to a complex place $F\hra \C$ and get $F\otimes_{F^+}\R\cong \C$.
Hence we have $Z^+_{F^+}(\R)=Z_F(\C)\neq \emptyset$.
Therefore we can apply (i) to the triple $(F^+, Z^+_{F^+}, V^+_{F^+})$ 
and find a totally real extension $L^+$ of $F^+$ such that 
$V_F(L^+\otimes_{F^+}F)=V^+_{F^+}(L^+)\neq\emptyset$.
Since $L^+\otimes_{F^+}F$ is a CM extension of $F$, this completes the proof.
\end{proof}

This lemma leads to the following definitions.

\begin{defn}
A \emph{totally real curve} is an open subscheme of the spectrum of the ring of integers of a totally real field.
A \emph{CM curve} is an open subscheme of the spectrum of the ring of integers of a CM field.
\end{defn}

\begin{defn}
Let $K$ be a totally real field and $X$ an irreducible regular $\calO_K$-scheme.
We say that $X$ is an $\calO_K$-scheme \emph{with enough totally real curves} 
if $X$ is flat and of finite type over $\calO_K$ with geometrically irreducible generic fiber and $X_K(\R)\neq \emptyset$ for every real place $K\hra \R$.
\end{defn}

Now we introduce some notation and state our approximation theorems.

\begin{defn}
Let $g\colon X\ra Y$ be a morphism of schemes.
For $x\in X$, consider the tangent space $T_x X=\Hom_{k(x)}\bigl(\fkm_x/\fkm^2_x, k(x)\bigr)$ at $x$, where $\fkm_x$ denotes the maximal ideal of the local ring at $x$. This contains $T_x(X_{g(x)})$, where $X_{g(x)}=X\otimes_Yk\bigl(g(x)\bigr)$.
A one-dimensional subspace $l$ of $T_x X$ is said to be \emph{horizontal} 
(with respect to $g$) if $l$ does not lie in the subspace $T_x(X_{g(x)})$.
\end{defn}

\begin{defn}
Let $X$ be a connected scheme and $Y$ a generically \'etale $X$-scheme. 
A point $x\in X(L)$ with some field $L$ is said to \emph{be inert in} $Y\ra X$ if for each irreducible component $Y_\alpha$ of $Y$, $(\Spec L)\times_{x, X}Y_\alpha$ is nonempty and connected.
\end{defn}

\begin{thm} 
\label{thm:Dr-thm2.15}
Let $K$ be a totally real field and $X$ an irreducible smooth separated $\calO_K$-scheme with enough totally real curves.
Consider the following data:
\begin{enumerate}
 \item a flat $\calO_K$-scheme $Y$ which is generically \'etale over $X$;
 \item a finite subset $S\subset \lvert X\rvert$ such that  $S\ra\Spec \calO_K$ is injective;
 \item a one-dimensional subspace $l_s$ of $T_sX$ for every $s\in S$.
\end{enumerate}

Then there exist a totally real curve $C$ with fraction field $L$, a morphism $\varphi\colon C\ra X$ and a section $\sigma\colon S\ra C$ of $\varphi$ over $S$ such that $\varphi(\Spec L)$ is inert in $Y\ra X$  and  $\Image (T_{\sigma(s)}C\ra T_sX)=l_s$ for every $s\in S$.
\end{thm}

\begin{proof}
We will use the theorem of Moret-Bailly to find the desired curve.
Note that we can replace $Y$ by any dominant \'etale $Y$-scheme. 

Let $v$ denote the structure morphism $X\ra \Spec \calO_K$.
Take an open subscheme $U\subset X$ 
such that $v(U)\cap v(S)=\emptyset$ and the morphism $Y\times_X U\ra U$ is finite and \'etale.
Replacing each connected component of $Y\times_X U$ by its Galois closure, we may assume that each connected component of $Y\times_X U$ is Galois over $U$.
Write $Y\times_X U=\coprod_{1\leq i\leq k}W_i$ as the disjoint union of connected components and denote by $G_i$ the Galois group of the covering $W_i\ra U$.
Let $H_{i1},\ldots,H_{ir_i}$ be all the proper subgroups of $G_i$.

We will choose a quadruple of the form
\[
\bigl(X_K, 
\{v_{ij}\}_{i,j}\cup\{v(s)\}_{s\in S}\cup \{v_{\infty,i}\}_i,
\{K_{v_{ij}}\}\cup \{M_s\}\cup \{\R\},
\{V_{ij}\}\cup \{V_s\}\cup \{V_{\infty,i}\}\bigr).
\]

First we choose $v_{ij}$ and $V_{ij}$ 
($1\leq i\leq k, 1\leq j\leq r_i$) which 
control the inerting property.

\begin{claim}\label{claim:choice of v_ij}
There exist a finite set $\{v_{ij}\}_{i,j}$ of finite places of $K$
and a nonempty open subset $V_{ij}$ of $X(K_{v_{ij}})$ 
with respect to $K_{v_{ij}}$-topology for each $(i,j)$
such that

\begin{enumerate}
 \item for any finite extension $L$ of $K$ and any $L$-rational point
$x\in X(L)$, $x$ is inert in $Y\ra X$ if 
$L$ is $K_{v_{ij}}$-split and if
all the images of $x$ under the induced maps 
$X(L)\ra X(K_{v_{ij}})$ lie in $V_{ij}$, and 
 \item $v_{ij}$ are different from any element of $v(S)$
regarded as a finite place of $K$.
\end{enumerate}
\end{claim}

\begin{proof}[Proof of Claim~\ref{claim:choice of v_ij}]
For each $i=1,\ldots,k$ and $j=1,\ldots, r_i$,
let $\pi_{H_{ij}}$ denote the induced morphism $W_i/H_{ij}\ra X$ and
let $M_{ij}$ be the algebraic closure of $K$ 
in the field of rational functions of $W_i/H_{ij}$.
Then we have a canonical factorization $W_i/H_{ij}\ra \Spec \calO_{M_{ij}}\ra \Spec \calO_K$.

If $M_{ij}=K$, then the generic fiber $(W_i/H_{ij})_K$ is geometrically integral over $K$.
It follows from Proposition~3.5.2 of \cite{MR2363329} that there are infinitely many finite places $v_0$ of $K$ such that $U(K_{v_0})\setminus \pi_{H_{ij}}\bigl(W_i/H_{ij}(K_{v_0})\bigr)$ is a nonempty open subset of $U(K_{v_0})$.
Thus choose such a finite place $v_{ij}$ and put 
\[
 V_{ij}=U(K_{v_{ij}})\setminus \pi_{H_{ij}}\bigl(W_i/H_{ij}(K_{v_{ij}})\bigr).
\]

Next consider the case where $M_{ij}\neq K$.
The Lang-Weil theorem and the Chebotarev density theorem show that there are infinitely many finite places $v_0$ of $K$ such that $U(K_{v_0})\neq \emptyset$ and $v_0$ does not split completely in $M_{ij}$, that is, $M_{ij}\otimes_KK_{v_0}\not\cong K_{v_0}^{[M_{ij}:K]}$ 
(see Propositions~3.5.1 and~3.6.1 of \cite{MR2363329} for example).
In this case, choose such a finite place $v_{ij}$ and put 
\[
 V_{ij}=U(K_{v_{ij}}).
\]

It is obvious to see that we can choose $v_{ij}$ satisfying condition (ii).
We now show that these $v_{ij}$ and $V_{ij}$ satisfy condition (i).
Take $L$ and $x\in X(L)$ as in condition (i).
By the lemma below (Lemma \ref{lem:inert criterion for a cover}), it suffices to prove that $x\not\in \pi_{H_{ij}}\bigl(W_i/H_{ij}(L)\bigr)$ for any $H_{ij}$.

When $M_{ij}=K$, this is obvious because the images of $x$ under the maps $X(L)\ra X(K_{v_{ij}})$ lie in $V_{ij}=U(K_{v_{ij}})\setminus \pi_{H_{ij}}\bigl(W_i/H_{ij}(K_{v_{ij}})\bigr)$.
When $M_{ij}\neq K$, we know that $M_{ij}\otimes_KK_{v_{ij}}$ is not $K_{v_{ij}}$-split. Since $L$ is assumed to be $K_{v_{ij}}$-split, $M_{ij}$ cannot be embedded into $L$.
On the other hand, we have a canonical factorization $W_i/H_{ij}\ra \Spec \calO_{M_{ij}}$.
Therefore $W_i/H_{ij}(L)=\emptyset$. 
Thus $x$ is inert in $Y\ra X$ in both cases.
\end{proof}

Next we choose a finite Galois extension $M_s$ of $K_{v(s)}$ and a $\Gal(M_s/K_{v(s)})$-stable nonempty open subset $V_s$ of $X(M_s)$ with respect to $M_s$-topology to make a totally real curve pass through $s$ in the tangent direction $l_s$. 
Here $K_{v(s)}$ denotes the completion of $K$ with respect to the finite place $v(s)$ of $K$.
Let $\hat{\calO}_{X,s}$ denote the completed local ring of $X$ at $s\in S$. 
Since $\hat{\calO}_{X,s}$ is regular, we can find a regular one-dimensional closed subscheme $\Spec R_s \subset \Spec \hat{\calO}_{X,s}$ which is tangent to $l_s$ and satisfies $R_s\otimes_{\calO_K}K\neq\emptyset$ (see  Lemma A.6 of \cite{MR3024821}). 

It follows from the construction that $R_s$ is a complete discrete valuation ring which is finite and flat over $\calO_{K_{v(s)}}$ and has residue field $k(s)$.
Let $M'_s$ be the fraction field of $R_s$.
For each $s\in S$ we first choose $M_s$ and a local homomorphism $\calO_{X,s}\ra \calO_{M_s}$. There are two cases.

If $l_s$ is horizontal, then $R_s$ is unramified over $\calO_{K_{v(s)}}$ and hence $M'_s$ is Galois over $K_{v(s)}$. Put $M_s:=M'_s$ in this case.
Then we have a natural local homomorphism
$\calO_{X,s}\ra\hat{\calO}_{X,s}\ra \calO_{M_s}$.

If $l_s$ is not horizontal, then $R_s$ is ramified over $\calO_{K_{v(s)}}$.
Let $K'_{v(s)}$ be the maximal unramified extension of $K_{v(s)}$ in $M'_s$ 
and $M_s$ the Galois closure of $M'_s$ over $K_{v(s)}$. 
Then both $K'_{v(s)}$ and $M_s$ have the same residue field $k(s)$.

We construct a local homomorphism
$\hat{\calO}_{X,s}\ra \calO_{M_s}$ in this setting.
Since $X$ is smooth over $\calO_K$, the ring $\hat{\calO}_{X,s}$ is isomorphic to 
the ring of formal power series  $\calO_{K'_{v(s)}}[[t_1,\ldots,t_m]]$ for some $m$ and we identify these rings.

Let $u_i\in \calO_{M'_s}$ denote the image of $t_i$ under the homomorphism
$\calO_{K'_{v(s)}}[[t_1,\ldots,t_m]] =\hat{\calO}_{X,s}\ra R_s=\calO_{M'_s}$. 
Let $\pi$ (resp.~ $\varpi$) be a uniformizer of $\calO_{M'_s}$ (resp.~ $\calO_{M_s}$) and consider $\pi$-adic expansion $u_i=\sum_{j=0}^\infty a_{ij}\pi^j$.
Since $\hat{\calO}_{X,s}\ra\calO_{M'_s}$ is a local homomorphism, we have $a_{i0}=0$ for each $i$. 

Consider the differential of $\Spec \calO_{M'_s}=\Spec R_s\ra \Spec \hat{\calO}_{X,s}$ at the closed point. The tangent vector $\frac{\partial}{\partial \pi}$ is sent to $\sum_{i=1}^ma_{i1}\frac{\partial}{\partial t_i}$ under this map, and the latter spans the tangent line $l_s$.

Define a local homomorphism
$\hat{\calO}_{X,s}\ra \calO_{M_s}$ by sending $t_i$
to $\sum_{j=1}^\infty a_{ij}\varpi^j$.  
Then the image of the differential of the corresponding morphism
$\Spec \calO_{M_s}\ra X$ at the closed point is $l_s$ by the same computation as above.

In either case, we have chosen $M_s$ and a homomorphism $\calO_{X,s}\ra \calO_{M_s}$.
Let $\hat{s}\in X(\calO_{M_s})$ be the point induced by 
the homomorphism.
Note that $X(\calO_{M_s})$ is an open subset of $X(M_s)$ by separatedness.
Let $\alpha\colon X(\calO_{M_s})\ra X(\calO_{M_s}/\fkm_{M_s}^2)$ be the reduction map, where $\fkm_{M_s}$ denotes the maximal ideal of $\calO_{M_s}$.
Define $V'_s=\alpha^{-1}\bigl(\alpha(\hat{s})\bigr)$, which is a nonempty open subset of $X(M_s)$,  
and put 
\[
 V_s=\bigcup_{\sigma}\sigma(V'_s),
\]
 where 
$\sigma$ runs over all the elements of $\Gal(M_s/K_{v(s)})$.
Since $\Gal(M_s/K_{v(s)})$ acts continuously on $X(M_s)$, 
$V_s$ is a nonempty $\Gal(M_s/K_{v(s)})$-stable open subset of $X(M_s)$.

Finally, let $v_{\infty,1},\ldots,v_{\infty,n}$  be the real places of $K$ and put 
\[
 V_{\infty,i}=X(\R)
\]
 for each $i=1,\ldots,n$.
This is nonempty by our assumption.

It follows from the theorem of Moret-Bailly (Theorem~\ref{thm:Moret-Bailly}) 
that there exist a finite extension $L$ of $K$ and an $L$-rational point $x\in X(L)$ satisfying the following properties:
\begin{enumerate}
 \item $L\otimes_KK_{v_{ij}}$ is $K_{v_{ij}}$-split and $x$ goes into $V_{ij}$
under any embedding $L\hra K_{v_{ij}}$.
 \item  $L\otimes_KM_s$ is $M_s$-split and $x$ goes into $V_s$
under any embedding $L\hra M_s$.
 \item $L$ is totally real.
\end{enumerate}

We can spread out the $L$-rational point $x\colon \Spec L\ra X$ to a morphism $\varphi\colon C\ra X$ where $C$ is a totally real curve with fraction field $L$. By property (ii), we can choose $C$ and $\varphi$ so that all the points of $\Spec \calO_L$ above $v(S)\subset \Spec \calO_K$ are contained in $C$. 
Claim~\ref{claim:choice of v_ij} shows that $x$ is inert in $Y\ra X$.
Thus it remains to prove that there exists a section $\sigma$ of $\varphi$ over $S$ such that $\Image (T_{\sigma(s)}C\ra T_sX)=l_s$ for every $s\in S$

It follows from property (ii) and the definition of $V_s$ that 
there exists an embedding $L\hra M_s$ such that
the image of $x$ under the associated map $X(L)\ra X(M_s)$ lies in $V'_s$.
Let $s'\in \Spec \calO_L$ be the closed point corresponding to this embedding.
Then we have $s'\in C$, $k(s')=k(s)$, and 
$\Image (T_{s'}C\ra T_sX)=l_s$.
Hence we can define a desired section of $\varphi$ over $S$.
\end{proof}

\begin{lem} \label{lem:inert criterion for a cover}
Let $L$ be a field, U a locally noetherian connected scheme and $\pi\colon W\ra U$ a Galois covering with Galois group $G$.
For any subgroup $H\subset G$, let $\pi_H$ denote the induced morphism $W/H\ra U$.
An $L$-valued point of $X$ is inert in $\pi$ if and only if it lies in $U(L)\setminus \bigcup_{H\subsetneq G}\pi_H\bigl(W/H(L)\bigr)$.
\end{lem}

\begin{proof}
Let $x$ denote the $L$-valued point.
Choose a point of $W$ above $x$ and fix a geometric point above it.
This also defines a geometric point above $x$ 
and we have a homomorphism $\pi_1(x)\ra G$, where
$\pi_1(x)$ is the absolute Galois group of $L$.
Let $H_0$ denote the image of this homomorphism.
Then $x$ is inert in $\pi$ if and only if 
the homomorphism is surjective, that is, $H_0=G$.
On the other hand, for a subgroup $H\subset G$, 
the point $x$ lies in $\pi_H\bigl(W/H(L)\bigr)$ 
if and only if $\Spec L\times_{x,U}W/H\ra\Spec L$ has a section,
which is equivalent to the condition that
some conjugate of $H$ contains $H_0$.
The lemma follows from these two observations.
\end{proof}

For our applications, we need a stronger variant of the theorem.

\begin{cor}  \label{cor:Dr-prop2.17}
Let $K$ be a totally real field and $X$ an irreducible smooth separated  $\calO_K$-scheme
with enough totally real curves.
Let $U$ be a nonempty open subscheme of $X$.
Suppose that we are given the following data:
\begin{enumerate}
 \item a flat $\calO_K$-scheme $Y$ which is generically \'etale over $X$;
 \item a closed normal subgroup $H\subset \pi_1(U)$ such that $\pi_1(U)/H$ contains an open pro-$\ell$ subgroup;
 \item a finite subset $S\subset\lvert X\rvert$ such that $S\ra\Spec \calO_K$ is injective;
 \item a one-dimensional subspace $l_s$ of $T_sX$ for every $s\in S$.
\end{enumerate}

Then there exist a totally real curve $C$ with fraction field $L$, 
a morphism $\varphi\colon C\ra X$ with $\varphi^{-1}(U)\neq\emptyset$ and a section $\sigma\colon S\ra C$ of $\varphi$ over $S$ such that
\begin{itemize}
 \item $\varphi(\Spec L)$ is inert in $Y\ra X$,
 \item $\pi_1\bigl(\varphi^{-1}(U)\bigr)\ra \pi_1(U)/H$ is surjective, and 
 \item $\Image (T_{\sigma(s)}C\ra T_sX)=l_s$ for every $s\in S$.
\end{itemize}
\end{cor}

\begin{proof}
As is shown in the proof of Proposition~2.17 of \cite{MR3024821}, we can find an open normal subgroup $G_0\subset \pi_1(U)/H$ satisfying the following property:
Every closed subgroup $G\subset \pi_1(U)/H$ such that the map $G\ra \bigl(\pi_1(U)/H\bigr)/G_0$ is surjective equals $\pi_1(U)/H$. 

Let $Y'$ be the Galois covering of $U$ corresponding to $G_0$.
Then we can apply Theorem~\ref{thm:Dr-thm2.15} to $\bigl(Y\sqcup Y', S, (l_s)_{s\in S}\bigr)$ and get the desired triple $(C, \varphi,\sigma)$.
\end{proof}

We have a similar approximation theorem in the CM case.
The proof uses the Weil restriction and is essentially similar to the totally real case, although one has to check that the conditions are preserved under the Weil restriction.

\begin{thm} \label{thm:Dr-thm2.15 for CM}
Let $F$ be a CM field and $Z$ an irreducible smooth separated $\calO_F$-scheme with geometrically irreducible generic fiber. Let $U$ be a nonempty open subscheme of $Z$.
Suppose that we are given the following data:
\begin{enumerate}
 \item a flat $\calO_F$-scheme $W$ which is generically \'etale over $Z$;
 \item a closed normal subgroup $H\subset \pi_1(U)$ such that 
$\pi_1(U)/H$ contains an open pro-$\ell$ subgroup;
 \item a finite subset $S\subset \lvert Z\rvert$ such that 
$S\ra\Spec \calO_{F^+}$ is injective;
 \item a one-dimensional subspace $l_s$ of $T_sZ$ for every $s\in S$.
\end{enumerate}
Then there exist a CM curve $C$ with fraction field $L$, 
a morphism $\varphi\colon C\ra Z$ with $\varphi^{-1}(U)\neq\emptyset$ and a section $\sigma\colon S\ra C$ of $\varphi$ over $S$ such that
\begin{itemize}
 \item $\varphi(\Spec L)$ is inert in $W\ra Z$,
 \item $\pi_1\bigl(\varphi^{-1}(U)\bigr)\ra \pi_1(U)/H$ is surjective, and
 \item $\Image (T_{\sigma(s)}C\ra T_sZ)=l_s$ for every $s\in S$.
\end{itemize}

\end{thm}

\begin{proof}
Let $F^+$ be the maximally totally real subfield of $F$.
Let $w$ (resp.~ $v$) denote the structure morphism $Z\ra \Spec\calO_F$
(resp.~ $Z\ra \Spec \calO_{F^+}$).
As in the proof of Corollary~\ref{cor:Dr-prop2.17}, we may omit 
the datum (ii) by replacing $W$ by another flat, generically \'etale $Z$-scheme and prove the first and third properties of the triple $(C,\varphi,\sigma)$.

Define
$Z^+$ to be the Weil restriction $\Res_{\calO_F/\calO_{F^+}}Z$.
Then we have
 $Z^+\otimes_{\calO_{F^+}}\calO_F\cong Z\times_{\calO_F}c^\ast Z$,
where $c$ denotes the nontrivial element of $\Gal(F/F^+)$ and
$c^\ast Z$ denotes $Z\otimes_{\calO_{F},c}\calO_F$.
It follows from the assumptions that $Z^+$ is 
an irreducible smooth $\calO_{F^+}$-scheme
with enough totally real curves.
We will apply the theorem of Moret-Bailly to $Z^+$ with appropriate data.

We may assume that each connected component of $W$ is a Galois cover over its image in $Z$ by replacing $W$ if necessary.
Put $Y=W\times_{\calO_F}c^\ast W$ and regard it as an $\calO_{F^+}$-scheme. 
Then $Y\ra Z\times_{\calO_F}c^\ast Z\ra Z^+$ 
is flat and generically \'etale, and therefore 
satisfies the same assumptions as $Y\ra X$ in Theorem~\ref{thm:Dr-thm2.15}
and the second paragraph of its proof.
Hence, as in Claim~\ref{claim:choice of v_ij} in the proof of Theorem~\ref{thm:Dr-thm2.15}, 
there exist a finite set $\{v_{ij}\}_{1\leq i\leq k, 1\leq j\leq r_i}$ of finite places of $F^+$
and a nonempty open subset $V_{ij}$ of $Z^+(F^+_{v_{ij}})$ for each $(i,j)$
satisfying the following properties:

\begin{enumerate}
 \item For any finite extension $L^+$ of $F^+$ and any $L^+$-rational point
$z^+\in Z^+(L^+)$, $z^+$ is inert in $Y\ra Z^+$ if 
$L^+$ is $F^+_{v_{ij}}$-split and if
$z^+$ lands in $V_{ij}$ under any embedding $L^+\hra F^+_{v_{ij}}$.
 \item $v_{ij}$ are different from any element of $v(S)$.
\end{enumerate}

Next take any $s\in S$. 
We will choose a finite Galois extension $M_s$ of $F_{w(s)}$
and a $\Gal(M_s/F^+_{v(s)})$-stable nonempty subset
of $Z^+(M_s)$ with respect to $M_s$-topology.
Here we regard $w(s)$ (resp.~ $v(s)$) as a finite place of $F$ (resp.~ $F^+$).
Then $M_s$ is Galois over $F^+_{v(s)}$ since
$w(s)$ lies above $v(s)$ and $[F_{w(s)}: F^+_{v(s)}]$ is either 1 or 2.

As in the proof of Theorem~\ref{thm:Dr-thm2.15}, 
we can find a finite Galois extension $M_s$ of $F_{w(s)}$ with residue field $k(s)$ and a homomorphism $\calO_{Z,s}\ra \calO_{M_s}$
such that the image of the differential of the corresponding morphism
$\Spec \calO_{M_s}\ra Z$ at the closed point is $l_s$.
Denote by $\hat{s}\in Z(\calO_{M_s})\subset Z(M_s)$ the point corresponding to
this morphism.

Let $\alpha\colon Z(\calO_{M_s})\ra Z(\calO_{M_s}/\fkm_{M_s}^2)$ be the reduction map, where $\fkm_{M_s}$ denotes the maximal ideal of $\calO_{M_s}$.
Define $V'_s=\alpha^{-1}\bigl(\alpha(\hat{s})\bigr)$, which is a nonempty open subset of $Z(M_s)$, 
and put $V''_s=\bigcup_{\sigma}\sigma(V'_s)$
where $\sigma$ runs over all the elements of $\Gal(M_s/F_{w(s)})$.

Denote by $\iota$ a natural embedding $F\hra F_{w(s)}\hra M_s$.
We have $\Hom_{F^+}(F,M_s)=\{\iota,\iota\circ c\}$ and
the $F^+$-homomorphism $(\iota, \iota\circ c)\colon F\ra M_s\times M_s$
induces an isomorphism $M_s\otimes_{F^+}F\cong M_s\times M_s$
which sends $a\otimes b$ to $\bigl(a\iota(b), a\iota(c(b))\bigr)$.
Hence we get identifications
$Z^+(M_s)=Z(M_s\otimes_{F^+}F)=Z(M_s)\times c^\ast Z(M_s)$. 
Here $Z(M_s)$ denotes $\Hom_{\calO_F}(\Spec M_s, Z)$ by regarding $\Spec M_s$ as an $F$-scheme via $\iota$.

Define 
\[
 V_s:=V_s''\times c^\ast V_s''\subset Z(M_s)\times c^\ast Z(M_s)=Z^+(M_s). 
\]
This is a nonempty open subset of $Z^+(M_s)$.
Since $V_s''$ is $\Gal(M_s/F_{w(s)})$-stable and $\Gal(F/F^+)=\{\id,c\}$, 
$V_s$ is $\Gal(M_s/F^+_{v(s)})$-stable.

Let $v_{\infty,1},\ldots, v_{\infty,n}$ be the real places of $F^+$.
Then for each $1\leq i\leq n$ put 
\[
V_{\infty,i}=Z^+(\R)=Z(\C)
\]
via an isomorphism $F\otimes_{F^+}\R\cong \C$. 
This is a nonempty open set.

We apply Theorem~\ref{thm:Moret-Bailly} to the quadruple
\[
\bigl(Z^+_{F^+}, 
\{v_{ij}\}_{i,j}\cup\{v(s)\}_{s\in S}\cup \{v_{\infty,i}\}_i,
\{F^+_{v_{ij}}\}\cup \{M_s\}\cup \{\R\},
\{V_{ij}\}\cup \{V_s\}\cup \{V_{\infty,i}\}\bigr)
\]
and find a totally real finite extension $L^+$ of $F^+$ and an $L^+$-rational point $z^+\in Z^+(L^+)$ satisfying the following properties:
\begin{enumerate}
 \item $z^+$ is inert in $Y\ra Z^+$.
 \item  $L^+$ is $M_s$-split and $z^+$ goes into $V_s$ via any embedding $L^+\hra M_s$.
\end{enumerate}

Let $L$ be the CM field $L^+\otimes_{F^+}F$ and $z\in Z(L)$ be
the $L$-rational point corresponding to $z^+\in Z^+(L^+)$.
Then the morphism $z$ is equal to the composite
\[
 \pr_Z\circ (z^+\otimes_{F^+}F)\colon \Spec L\ra Z^+\otimes_{\calO_{F^+}}\calO_F
=Z\times_{\calO_F}c^\ast Z\ra Z.
\]
We can spread out $z\colon \Spec L\ra Z$ to a morphism
$\varphi\colon C\ra Z$ 
for some CM curve $C$ with fraction field $L$.
We may assume that $C$ contains all the points of $\Spec \calO_L$ above $w(S)\subset \Spec \calO_F$.
It follows from property (ii) and the definition of $V_s$ that $\varphi$ has a section $\sigma$ over $S$ such that $\Image (T_{\sigma(s)}C\ra T_sX)=l_s$ for every $s\in S$.

It remains to prove that $z=\varphi(\Spec L)$ is inert in $W\ra Z$.
Without loss of generality, we may assume that $W_F$ is connected, and thus
it suffices to show that 
$\Spec L\times_{z,Z}W=\Spec L\times_{z,Z_F}W_F$ is connected.
Define the schemes $P$ and $Q$ such that 
the squares in the following diagram are Cartesian:
\[
\xymatrix{
P \ar[r] \ar[d] & Q \ar[r] \ar[d]& \Spec L \ar[r]\ar[d]^{z^+\otimes_{F^+}F}
& \Spec L^+ \ar[d]^{z^+}\\
W_F\times_Fc^\ast W_F \ar[r] & W_F\times_Fc^\ast Z_F \ar[r]\ar[d]
& Z_F\times_Fc^\ast Z_F \ar[r] \ar[d]^{\pr_{Z_F}} & Z^+_{F^+} \\
&W_F\ar[r]&Z_F.
} 
\]

Since $W_F\times_Fc^\ast W_F=Y_{F^+}$, we have
$P\cong \Spec L^+\times_{z^+,Z^+_{F^+}}Y_{F^+}$.
As $Q\cong \Spec L\times_{z,Z_F}W_F$,
we need to show that $Q$ is connected.

Note that $W_F\times_Fc^\ast Z_F$ is connected;
this follows from the fact 
that $c^\ast Z_F$ is geometrically connected over $F$
and $W_F$ is connected.
Now take any connected component $T$ of $Y_{F^+}=W_F\times_F c^\ast W_F$.
Since $W_F\times_F c^\ast W_F\ra W_F\times_Fc^\ast Z_F$ is an 
 \'etale covering with connected base, $T$ surjects onto 
$W_F\times_Fc^\ast Z_F$.
At the same time, the subscheme $\Spec L^+\times_{z^+,Z^+_{F^+}}T\subset P$ 
is connected because $z^+$ is inert in $Y\ra Z^+$.
Since the connected scheme $\Spec L^+\times_{z^+,Z^+_{F^+}}T$ surjects 
onto $Q$, the latter is also connected.
\end{proof}

\begin{rem}\label{rem:sm and sep assump for Dr-thm2.15}
In Theorem~\ref{thm:Dr-thm2.15}, Corollary~\ref{cor:Dr-prop2.17}, and Theorem~\ref{thm:Dr-thm2.15 for CM}, we assume that the scheme in question is smooth and separated. If $S=\emptyset$, then we can replace these two assumptions by regularity.
In fact, if $S=\emptyset$, we can replace the scheme by an open subscheme, and thus reduce to the separated case. Moreover, the regularity implies that the generic fiber of the scheme is smooth. So we can apply the theorem of Moret-Bailly to our scheme. Note that the smoothness assumption was used only when $S\neq \emptyset$ and $l_s$ is not horizontal for some $s\in S$.
\end{rem}

\section{Proofs of Theorem~\ref{thm:Dr-thm2.5} and its variants}
\label{section:prf of Dr-thm2.5}

In this section, we prove Theorem~\ref{thm:Dr-thm2.5} and its variants following \cite{MR3024821}.
First we set up our notation.
Fix a prime $\ell$ and a finite extension $E_\lambda$ of $\Q_\ell$.
Let $\calO$ be the ring of integers of $E_\lambda$ and $\fkm$ its maximal ideal.

Fix a positive integer $r$. For a normal scheme $X$ of finite type over $\Spec \Z[\ell^{-1}]$, $\LS^{E_\lambda}_r(X)$ denotes the set of equivalence classes of lisse $E_\lambda$-sheaves on $X$ of rank $r$, and $\widetilde{\LS}^{E_\lambda}_r(X)$ denotes the set of maps from the set of closed points of $X$ to the set of polynomials of the form $1+c_1t+\cdots+c_rt^r$ with $c_i\in\calO$ and $c_r\in \calO^\times$. 
Here we say that two lisse $E_\lambda$-sheaves on $X$ are \emph{equivalent} if they have isomorphic semisimplifications. 
Since the coefficient field $E_\lambda$ is fixed throughout this section, 
we simply write $\LS_r(X)$ or $\widetilde{\LS}_r(X)$.

For an element $f\in \widetilde{\LS}_r(X)$, we denote by $f_x(t)$ or $f(x)(t)$ the value of $f$ at $x\in X$; this is a polynomial in $t$.
By the Chebotarev density theorem, we can regard $\LS_r(X)$ as a subset of $\widetilde{\LS}_r(X)$ by attaching to each equivalence class its Frobenius characteristic polynomials. 
For another scheme $Y$ and a morphism $\alpha\colon Y\ra X$, we have a canonical map $\alpha^\ast\colon \widetilde{\LS}_r(X)\ra\widetilde{\LS}_r(Y)$ whose restriction to $\LS_r(X)$ coincides with the pullback map of sheaves 
$\LS_r(X)\ra \LS_r(Y)$. 
We also denote $\alpha^\ast (f)$ by $f|_Y$.

Let $C$ be a separated smooth curve over a finite field and $\overline{C}$ the smooth compactification of $C$. 
We define $\LS^{\tame}_r(C)$ to be the subset of $\LS_r(C)$ consisting of equivalence classes of lisse $E_\lambda$-sheaves on $C$ which are tamely ramified at each point of $\overline{C}\setminus C$. This condition does not depend on the choice of a lisse sheaf in the equivalence class.
Let $\varphi$ be a morphism $C\ra X$ and $f\in \widetilde{\LS}_r(X)$.
When $\varphi^\ast(f) \in \LS_r(C)$ (resp.~ $\varphi^\ast(f) \in \LS^{\tame}_r(C)$), we simply say that $f$ \emph{arises from} a lisse sheaf (resp.~ a tame lisse sheaf) over the curve $C$.

To show Theorem 2.5 of \cite{MR3024821}, which is a prototype of Theorem~\ref{thm:Dr-thm2.5}, 
 Drinfeld considers a subset $\LS'_r(X)$  of $\widetilde{\LS}_r(X)$ which contains $\LS_r(X)$ and is characterized by a group-theoretic property.
He then proves the following three statements for $f\in \widetilde{\LS}_r(X)$,
which imply Theorem 2.5 of \cite{MR3024821}.

\begin{itemize}
 \item If the map $f$ satisfies two conditions\footnote{One needs tameness assumption in the second condition (it is identical to condition (ii) of Propositions ~\ref{prop:Dr-prop2.12}).}
similar to those in Theorem~\ref{thm:Dr-thm2.5}, then $f|_U\in \LS'_r(U)$ for some dense open subscheme $U\subset X$. 
 \item If $U$ is regular, then $\LS'_r(U)=\LS_r(U)$. In particular, the restriction $f|_U\in \LS'_r(U)$ arises from a lisse sheaf.
 \item If $f|_U$ arises from a lisse sheaf, then so does $f$ 
under the assumptions that $X$ is regular and that $f|_C$ arises from a lisse sheaf for every regular curve $C$.
\end{itemize}

Following Drinfeld, we will introduce the group-theoretic notion of ``having a kernel'' and prove similar statements, Propositions ~\ref{prop:Dr-prop2.12}, \ref{prop:Dr-prop2.13}, and \ref{prop:Dr-prop2.14}.
Theorem~\ref{thm:Dr-thm2.5} and its variants will be deduced from them at the end of the section.

\begin{defn}
Let $X$ be a scheme of finite type over $\Z[\ell^{-1}]$ and $f\in \widetilde{\LS}_r(X)$.
For a nonzero ideal $I\subset\calO$, the map $f$ is said to be \emph{trivial modulo $I$} if it has the value congruent to $(1-t)^r$ modulo $I$ at every closed point of $X$. 

When $X$ is connected, the map $f$ is said to \emph{have a kernel} if there exists a closed normal subgroup $H\subset \pi_1(X)$ satisfying the following conditions:
\begin{enumerate}
 \item $\pi_1(X)/H$ contains an open pro-$\ell$ subgroup.
 \item For every $n\in\N$, there exists an open subgroup $H_n\subset \pi_1(X)$ containing $H$ such that the pullback of $f$ to $X_n$ is trivial modulo $\fkm^n$. Here $X_n$ denotes the covering of $X$ corresponding to $H_n$.
\end{enumerate}
When $X$ is disconnected, the map $f$ is said to have a kernel if the restriction of $f$ to each connected component of $X$ has a kernel.
\end{defn}

\begin{rem} \label{rem:sheaf has a kernel}
If $f$ arises from a lisse sheaf on $X$, it has a kernel. 
To see this, we may assume that $X$ is connected.
Then the kernel of the $E_\lambda$-representation of $\pi_1(X)$ corresponding to the lisse sheaf satisfies the conditions.
\end{rem}

\begin{rem}
 The set $\LS'_r(X)$ defined by Drinfeld
consists of the maps $f$ which have a kernel and arise from a lisse sheaf over every regular curve (Definition~2.11 of \cite{MR3024821}).
\end{rem}

\begin{prop} \label{prop:Dr-prop2.12}
Let $K$ be a totally real field.
Let $X$ be an irreducible regular $\calO_K[\ell^{-1}]$-scheme with enough totally real curves and $f\in \widetilde{\LS}_r(X)$.
Assume that
 \begin{enumerate}
  \item $f$ arises from a lisse sheaf over every totally real curve, and
  \item there exists a dominant \'etale morphism $X'\ra X$ such that
the pullback $f|_{X'}$ arises from a tame lisse sheaf over every separated smooth curve over a finite field.
 \end{enumerate}
Then there exists a dense open subscheme $U\subset X$ such that $f|_U$ has a kernel.
\end{prop}

We will first show two lemmas and then prove Proposition~\ref{prop:Dr-prop2.12} by induction on the dimension of $X$.
For this we use elementary fibrations, which we recall now.

\begin{defn}
A morphism of schemes $\pi\colon X\ra S$ is called an \emph{elementary fibration} 
if there exist an $S$-scheme $\pi\colon \overline{X}\ra S$ and a factorization
$X\ra \overline{X}\stackrel{\overline{\pi}}{\ra} S$ of $\pi$
such that
\begin{enumerate}
 \item the morphism $X\ra \overline{X}$ is an open immersion and
$X$ is fiberwise dense in $\overline{\pi}\colon\overline{X}\ra S$,
 \item $\overline{\pi}$ is a smooth and projective morphism whose geometric fibers are nonempty irreducible curves, and
 \item the reduced closed subscheme $\overline{X}\setminus X$ is finite and \'etale over $S$.
\end{enumerate}
\end{defn}

The next lemma, which is due to Drinfeld and Wiesend, is a key to our induction argument in the proof of Proposition~\ref{prop:Dr-prop2.12}.
\begin{lem} \label{lem:induction in Dr-lem3.1}
Let $X$ be a scheme of finite type over $\Z[\ell^{-1}]$ and $f\in \widetilde{\LS}_r(X)$. 
Suppose that $X$ admits an elementary fibration $X\ra S$ with a section $\sigma\colon S\ra X$. 
Assume that
\begin{enumerate}
 \item $f$ arises from a tame lisse sheaf over every fiber of $X\ra S$, and
 \item there exists a dense open subscheme $V\subset S$ such that $\sigma^\ast(f)|_V$ has a kernel.
\end{enumerate}
Then there exists a dense open subscheme $U\subset X$ such that $f|_U$ has a kernel.
\end{lem}

\begin{proof}
This is shown in the latter part of the proof of Lemma~3.1 of \cite{MR3024821}. 
For convenience of the reader, we summarize the proof below.

We may assume that $X$ is connected and normal, and that $V=S$. 
For every $n\in \N$, consider the functor which attaches to an $S$-scheme $S'$ the set of isomorphism classes of $\GL_r(\calO/\fkm^n)$-torsors on $X\times_SS'$ tamely ramified along $(\overline{X}\setminus X)\times_SS'$ relative to $S'$ with trivialization over the section $S'\hra X\times_SS'$.
Then this functor is representable by an \'etale scheme $T_n$ of finite type over $S$ and the morphism $T_{n+1}\ra T_n$ is finite for each $n$.
By shrinking $S$, we may assume that the morphism $T_n\ra S$ is finite for each $n$. We will prove that $f$ has a kernel in this situation.

Since $\sigma^\ast(f)$ has a kernel by assumption (ii), there exist connected \'etale coverings $S_n$ of $S$ such that 
\begin{itemize}
 \item the pullback of $\sigma^\ast(f)$ to $S_n$ is trivial modulo $\fkm^n$, and
 \item for some (or any) geometric point $\bar{s}$ of $S$, the quotient of the group $\pi_1(S,\bar{s})$ by the intersection of the kernels of its actions on the fibers $(S_n)_{\bar{s}}$ where $n$ runs in $\N$ contains an open pro-$\ell$ subgroup.
\end{itemize}

Let $\fkT_n$ be the universal tame $\GL_r(\calO/\fkm^n)$-torsor over $X\times_ST_n$. Define the $X$-scheme $Y_n$ to be the Weil restriction $\Res_{X\times_ST_n/X}\fkT_n$ and let $X_n$ denote $Y_n\times_S S_n$. We thus have a diagram
whose squares are Cartesian
\[
\xymatrix{
X_n \ar[r] \ar[d] & X\times_SS_n \ar[r] \ar[d] & S_n \ar[ld] \ar[d]\\
Y_n \ar[r] & X \ar[r] & S\\
}
\]
 and regard $X_n$ as an \'etale covering of $X$. 
Here the morphism $S_n\ra X$ is the composite of $S_n\ra S$ and the section $\sigma\colon S\ra X$.

It suffices to prove the following two assertions:
\begin{enumerate}
 \item[(a)] The pullback of $f$ to $X_n$ is trivial modulo $\fkm^n$.
 \item[(b)] For some (or any) geometric point $\bar{x}$ of $X$, the quotient of the group $\pi_1(X,\bar{x})$ by the intersection of the kernels of its actions on the fibers $(X_n)_{\bar{x}}$ where $n$ runs in $\N$ contains an open pro-$\ell$ subgroup.
\end{enumerate}

In fact, if we take a Galois covering $X_n'$ of $X$ splitting the (possibly disconnected) covering $X_n$, the corresponding subgroup $H_n:=\pi(X_n',\bar{x})\subset \pi_1(X,\bar{x})$ and the intersection $H:=\bigcap_n H_n$ satisfy the conditions for the map $f$ to have a kernel.

First we prove assertion (a). Take an arbitrary closed point $x\in X_n$.
Let $s\in S$ denote the image of $x$ and choose a geometric point $\bar{s}$ above $s\in S$. 
By assumption (i), the restriction $f|_{X_s}$ arises from a lisse $E_\lambda$-sheaf of rank $r$. Let $\calF$ be a locally constant constructible sheaf of free $(\calO/\fkm^n)$-modules of rank $r$ obtaining from the above lisse sheaf modulo $\fkm^n$.

Consider the $X_{\bar{s}}$-scheme $(Y_n)_{\bar{s}}$.
The scheme $(X\times_ST_n)_{\bar{s}}$ is the disjoint union of copies of $X_{\bar{s}}$, and $(\fkT_n)_{\bar{s}}$ is a disjoint union of the $\GL_r(\calO/\fkm^n)$-torsors, each of which lies above a copy of $X_{\bar{s}}$ in $(X\times_ST_n)_{\bar{s}}$. 
Since the Weil restriction and the base change commute, $(Y_n)_{\bar{s}}$ is the fiber product of the tame $\GL_r(\calO/\fkm^n)$-torsors over $X_{\bar{s}}$. Hence $\calF|_{(Y_n)_{\bar{s}}}$ is constant, and so is $\calF|_{(X_n)_{\bar{s}}}$. 

Now let $s'\in S_n$ be the image of $x$. 
By the choice of $S_n$, we have  
\[
 (\sigma^\ast(f))|_{S_n}(s')(t)\equiv (1-t)^r \mod \fkm^n.
\]
Since we have shown that $\calF|_{(X_n)_{\bar{s}}}$ is constant, 
it follows that $f|_{X_n}(x)(t)\equiv (1-t)^r \mod \fkm^n$.

Finally, we prove assertion (b).
Let $\eta$ be the generic point of $S$. Choose a geometric point $\bar{\eta}$ above $\eta\in S$ and let $\bar{x}$ denote the geometric point above $\sigma(\eta)$ induced from $\bar{\eta}$.
Let $H$ be the intersection of the kernels of actions 
of $\pi_1(X,\bar{x})$
on the fibers $(X_n)_{\bar{x}}$ where $n$ runs in $\N$.
We need to show that $\pi_1(X,\bar{x})/H$
contains an open pro-$\ell$ subgroup.

Using the fact that the tame fundamental group $\pi_1^{\tame}(X_{\bar{\eta}},\bar{x})$ is topologically finitely generated, one can prove that the quotient of the group $\pi_1(X,\bar{x})$ by the intersection $H_Y$ of the kernels of its actions on the fibers $(Y_n)_{\bar{x}}, n\in \N$ contains an open pro-$\ell$ subgroup (see the last part of the proof of Lemma~3.1 of \cite{MR3024821}).

Let $H_S'$ be the intersection of the kernels of actions of $\pi_1(S,\bar{\eta})$ on the fibers $(S_n)_{\bar{\eta}}$ where $n$ runs in $\N$ and let $H_S$ be the inverse image of $H_S'$ with respect to the homomorphism $\pi_1(X,\bar{x})\ra \pi_1(S,\bar{\eta})$. 
By the choice of $S_n$, the group $\pi_1(S,\bar{\eta})/H_S'$ contains an open pro-$\ell$ subgroup. 
Since we have a surjection 
\[
 \pi_1(X,\bar{x})/(H_Y\cap H_S)\ra \pi_1(X,\bar{x})/H
\] and an injection
\[
 \pi_1(X,\bar{x})/(H_Y\cap H_S)\ra \pi_1(X,\bar{x})/H_Y\times \pi_1(S,\bar{\eta})/H_S',
\]
 the group $\pi_1(X,\bar{x})/H$ also contains an open pro-$\ell$ subgroup. 
\end{proof}

To use the above lemma, we show that there exists a chain of split fibrations ending with a totally real curve.

\begin{defn}
A sequence of schemes $X_n\ra X_{n-1}\ra\cdots\ra X_1$ 
is called \emph{a chain of split fibrations} if 
the morphism $X_{i+1}\ra X_i$ is an elementary fibration which admits a section $X_i\ra X_{i+1}$
for each $i=1,\ldots,n-1$ .
\end{defn}

\begin{lem} \label{lem:totally real elementary fibration}
Let $K$ be a totally real field and
$X$ an $n$-dimensional irreducible regular $\calO_K$-scheme 
with enough totally real curves.
Then there exist an \'etale $X$-scheme $X_n$, a totally real curve $X_1$ and a chain of split fibrations $X_n\ra \cdots\ra X_1$.
\end{lem}

\begin{proof}
We prove the lemma by induction on $n=\dim X$. When $\dim X=1$, the lemma holds by assumption. Thus we assume $\dim X\geq2$.

By induction on $\dim X$, it suffices to prove that after replacing $K$ by a totally real field extension and $X$ by a nonempty \'etale $X$-scheme,
there exist an irreducible regular $\calO_K$-scheme $S$ with enough totally real curves and an elementary fibration $X\ra S$ with a section $S\ra X$.

It follows from Lemma~\ref{lem:totally real points} (i) that there exists a totally real extension $L$ of $K$ such that $X_K(L)\neq\emptyset$.
Replacing $K$ by $L$ and $X$ by a nonempty open subscheme of $X\otimes_{\calO_K}\calO_L$ that is \'etale over $X$, we may further assume that the generic fiber $X_K\ra \Spec K$ has a section $x\colon \Spec K\ra X_K$. We also denote the image of $x$ in $X_K$ by $x$.

If $\dim X=2$, then $X_K$ is a smooth and geometrically connected curve over $K$. Take the smooth compactification $\overline{X}_K$ of $X_K$ over $K$.
Then the structure morphism $X_K\ra \Spec K$ has the factorization
$X_K\subset \overline{X}_K\ra \Spec K$ 
and thus it is an elementary fibration with a section $x$. 
After shrinking $X$, we can spread it out into an elementary fibration $X\ra S$ over an open subscheme $S$ of $\Spec \calO_K$ such that it admits a section $S\ra X$.

Now assume $\dim X\geq 3$. We apply Artin's theorem on elementary fibration (Proposition~3.3 in Expos\'{e} XI of \cite{MR0354654}) to the pair $(X_K, x)$, and by shrinking $X$ if necessary we get an elementary fibration $\pi\colon X_K\ra S_K$ over $K$ with a geometrically irreducible smooth $K$-scheme $S_K$.
Note that this theorem holds if the base field is perfect and infinite.

Since $X_K$ is smooth over $S_K$, there exist an open neighborhood $V_K$ of $x$ in $X_K$ and an \'etale morphism $\alpha\colon V_K\ra \A^1_{S_K}$ such that $\pi|_{V_K}\colon V_K\ra S_K$ admits a factorization 
\[
 V_K\stackrel{\alpha}{\lra}\A^1_{S_K}\ra S_K.
\]
Take a section $\tau\colon S_K\ra \A^1_{S_K}$ of the projection such that
$\alpha(x)$ lies in $\tau(S_K)$.

Consider  the connected component $S'_K$ of $S_K\times_{\tau,\A^1_{S_K}}V_K$ that contains the $K$-rational point $(\pi(x), x)$.
This is \'etale over $S_K$ and satisfies $S'_K(K)\neq \emptyset$. 
Moreover, $S'_K$ is geometrically integral over $K$ since it is a connected regular $K$-scheme containing a $K$-rational point.

We replace $S_K$ by  $S'_K$ and $X_K$ by $X_K\times_{S_K}S'_K$.
By this replacement, the elementary fibration $\pi\colon X_K\ra S_K$ admits a section and $S_K(K)\neq \emptyset$.
After shrinking $X$, we can spread it out into an elementary fibration $X\ra S$ with a section $S\ra X$, where $S$ is an irreducible regular scheme which is flat and of finite type over $\calO_K$ with geometrically irreducible generic fiber and contains a $K$-rational point.
The existence of a $K$-rational point implies that $S$ has enough totally real curves. 
\end{proof}

\begin{proof}[Proof of Proposition~\ref{prop:Dr-prop2.12}]
First note that if $\alpha^\ast(f)|_{U''}$ has a kernel for a nonempty \'{e}tale $X$-scheme $\alpha\colon X''\ra X$ 
and a dense open subscheme $U''\subset X''$, then so does $f|_{\alpha(U'')}$.

Let $n$ denote the dimension of $X$.
Replacing $X$ by the image of $X'\ra X$, we may assume that $X'\ra X$ is surjective.

Take a chain of split fibrations $X_n\ra X_{n-1}\ra\cdots\ra X_1$ with
a totally real curve $X_1$ as in Lemma~\ref{lem:totally real elementary fibration}.
We regard $X_1$ as an $X$-scheme via $X_n\ra X$ and the sections $X_i\ra X_{i+1}$.
Put $X'_1=X'\times_XX_1$. This is a nonempty scheme.
For $i=2,\ldots,n$ we put $X_i'=X_i\times_{X_1}X_1'$ via the morphism
$X_i\ra X_{i-1}\ra \cdots \ra X_1$.
Then $X_n'\ra X_{n-1}'\ra \cdots \ra X_1'$ is a chain of split fibrations.

Since $f|_{X_1}$ lies in $\LS_r(X_1)$ by assumption (i), we have $f|_{X'_1}=(f_{X_1})|_{X'_1}\in \LS_r(X'_1)$.
Then we get the result for $(X_2',f|_{X_2'})$ by Lemma~\ref{lem:induction in Dr-lem3.1}.
Repeating this argument for the chain of split fibrations
$X_n'\ra\cdots \ra X_2'$ we get the result for $(X_n',f|_{X_n'})$.
Applying the remark at the beginning to the morphism $X'_n\ra X$, we get the result for $(X,f)$.
\end{proof}

For the later use, we prove  variants of Lemma~\ref{lem:totally real elementary fibration}.
The proof given below is similar to that of Lemma~\ref{lem:totally real elementary fibration}, but instead of Lemma~\ref{lem:totally real points}
we will use Corollary~\ref{cor:Dr-prop2.17}
and Theorem~\ref{thm:Dr-thm2.15 for CM}.

\begin{lem}\label{lem:var of totally real elementary fibration}
\hfill
\begin{enumerate}
 \item With the notation as in Lemma~\ref{lem:totally real elementary fibration}, suppose that we are given a connected \'etale covering $Y\ra X$.
Then there exist an \'etale $X$-scheme $X_n$, a totally real curve $X_1$, and a chain of split fibrations $X_n\ra \cdots\ra X_1$ 
such that $X_1\times_{X}Y$ is connected.
Here $X_1\ra X$ is the composite of sections $X_{i+1}\ra X_i$ and $X_n\ra X$.
 \item Let $F$ be a CM field and
$Z$ an $n$-dimensional irreducible regular $\calO_F$-scheme with geometrically irreducible generic fiber. Let $Y\ra Z$ be a connected \'etale covering.
Then there exist an \'etale $Z$-scheme $Z_n$, a CM curve $Z_1$, and a chain of split fibrations $Z_n\ra \cdots\ra Z_1$ 
such that $Z_1\times_{Z}Y$ is connected.
Here $Z_1\ra Z$ is the composite of sections $Z_{i+1}\ra Z_i$ and $Z_n\ra Z$.
\end{enumerate}
\end{lem}

\begin{proof} 
First we prove (i) by induction on $n=\dim X$.
Since the claim is obvious when $\dim X=1$, we assume $\dim X\geq 2$.

By induction on $\dim X$, it suffices to prove that after replacing $K$ by a totally real field extension, $X$ by a nonempty \'etale $X$-scheme, and 
the covering $Y\ra X$ by its pullback, 
there exist an irreducible regular $\calO_K$-scheme $S$ with enough totally real curves and an elementary fibration $X\ra S$ with a section $S\ra X$ such that
$S\times_X Y$ is connected.
The construction of such an $S$ will be the same as that of Lemma~\ref{lem:totally real elementary fibration}.

It follows from Corollary~\ref{cor:Dr-prop2.17} 
and Remark~\ref{rem:sm and sep assump for Dr-thm2.15}
that there exist a totally real extension $L$ of $K$ 
and an $L$-rational point $x\in X(L)$
such that $\Spec L\times_{x,X}Y$ is connected.
Note that 
$Y\otimes_{\calO_K}\calO_L$ is connected
because $Y\otimes_{\calO_K}\calO_L\ra X\otimes_{\calO_K}\calO_L$
is an \'etale covering with connected base and 
\[
 \Spec L\times_{x,(X\otimes_{\calO_K}\calO_L)}(Y\otimes_{\calO_K}\calO_L)=\Spec L\times_{x,X}Y
\]
is connected.
Thus replacing $K$ by $L$, $X$ by a nonempty open subscheme of $X\otimes_{\calO_K}\calO_L$ that is \'etale over $X$, and $Y$ by its pullback, we may further assume that the generic fiber $X_K\ra \Spec K$ has a section $x\colon \Spec K\ra X_K$
such that $\Spec K\times_{x,X}Y$ is connected. We also denote the image of $x$ in $X_K$ by $x$.

If $\dim X=2$,  the morphism $X_K\ra \Spec K$ is an elementary fibration with a section $x$. 
After shrinking $X$, we can spread it out into an elementary fibration $X\ra S$ over an open subscheme $S$ of $\Spec \calO_K$ such that it admits a section $S\ra X$. By construction, $S\times_X Y$ is connected.

Now assume $\dim X\geq 3$. We apply Artin's theorem on elementary fibration to the pair $(X_K, x)$, and by shrinking $X$ if necessary we get an elementary fibration $\pi\colon X_K\ra S_K$ over $K$ with a geometrically irreducible smooth $K$-scheme $S_K$.

By smoothness, there exist an open neighborhood $V_K$ of $x$ in $X_K$ and an \'etale morphism $\alpha\colon V_K\ra \A^1_{S_K}$ such that $\pi|_{V_K}\colon V_K\ra S_K$ admits a factorization $V_K\stackrel{\alpha}{\ra}\A^1_{S_K}\ra S_K$.
Take a section $\tau\colon S_K\ra \A^1_{S_K}$ of the projection such that
$\alpha(x)$ lies in $\tau(S_K)$.

Consider the connected component $S'_K$ of $S_K\times_{\tau,\A^1_{S_K}}V_K$ that contains the $K$-rational point $(\pi(x), x)$.
As is shown in Lemma~\ref{lem:totally real elementary fibration}, $S_K'$ is \'etale over $S_K$ and geometrically integral over $K$, and $S_K'(K)\neq \emptyset$.

The section $\tau$ defines the morphism
$S'_K \ra X_K\times_{S_K}S'_K \ra X_K$.
The composite of this morphism and $(\pi(x),x)\colon \Spec K\ra S'_K$ coincides with $x\colon \Spec K\ra X_K$.
Since $S'_K\times_XY\ra S'_K$ is an \'etale covering with connected base and
\[
\Spec K\times_{(\pi(x),x),S_K'}(S_K'\times_XY) = \Spec K\times_{x, X}Y
\]
is connected, it follows that $S_K'\times_XY$ is connected.

We replace $S_K$ by $S'_K$, $X_K$ by $X_K\times_{S_K}S'_K$ and
$Y_K$ by $Y_K\times_{S_K}S'_K$.
By this replacement, the elementary fibration $\pi\colon X_K\ra S_K$ admits a section such that $S_K(K)\neq\emptyset$ and $S_K\times_{X_K}Y_K$ is connected.
As is discussed in the last paragraph of the proof of Lemma~\ref{lem:totally real elementary fibration}, after shrinking $X$, we can spread out $X_K\ra S_K$ and $Y_K\ra S_K$ into an elementary fibration $X\ra S$ with a section $S\ra X$ and a covering $Y\ra X$, where $S$ is an irreducible regular $\calO_K$-scheme with enough totally real curves.
Since $S\times_X Y$ is connected by construction, this $S$ works.

For part (ii),
it is easy to verify that the same argument works
if we apply Theorem~\ref{thm:Dr-thm2.15 for CM} instead of Corollary~\ref{cor:Dr-prop2.17}.
\end{proof}

Next we show that if $f$ has a kernel and arises from a lisse sheaf over every
totally real curve then it actually arises from a lisse sheaf.
\begin{prop} \label{prop:Dr-prop2.13}
Let $K$ be a totally real field and X an irreducible smooth separated
$\calO_K[\ell^{-1}]$-scheme with enough totally real curves.
Suppose that $f\in \widetilde{\LS}_r(X)$ satisfies the following conditions:
\begin{enumerate}
  \item $f$ arises from a lisse sheaf over every totally real curve.
  \item $f$ has a kernel.
\end{enumerate}
Then $f\in \LS_r(X)$.
\end{prop}

\begin{proof}
We follow Section 4 of \cite{MR3024821}.
Since $f$ has a kernel, we take a closed subgroup $H$ of $\pi_1(X)$
as in the definition of having a kernel.
In particular, $\pi_1(X)/H$ contains an open pro-$\ell$ subgroup.

By Corollary~\ref{cor:Dr-prop2.17}, there exists a totally real curve $C$ with a morphism $\varphi\colon C\ra X$ such that $\varphi_\ast\colon\pi_1(C)\ra \pi_1(X)/H$ is surjective.
By assumption (i), for any such pair $(C,\varphi)$, the pullback  $\varphi^\ast(f)$ arises from
a semisimple representation $\rho_C\colon \pi_1(C)\ra\GL_r(E_\lambda)$.
Define 
\[
H_C:=\Ker\bigl(\varphi_\ast\colon\pi_1(C)\ra\pi_1(X)/H\bigr).
\]

Then condition (ii) in the definition of having a kernel, together with the Chebotarev density theorem, shows that $\Ker \rho_C$ contains $H_C$.
See Lemma~4.1 of \cite{MR3024821} for details.
Thus we regard $\rho_C$ as a representation
\[
\rho_C\colon \pi_1(X)\ra\pi_1(X)/H\ra \GL_r(E_\lambda) 
\] 
of $\pi_1(X)$. Note that $\rho_C|_{\pi_1(C)}$ gives the original representation of $\pi_1(C)$.

We will show that the lisse sheaf on $X$ corresponding to this representation gives $f$. For this, we need to show that 
\[
 \det\bigl(1-t\rho_C(\Frob_x)\bigr)=f_x(t)
\]
for all closed points $x\in X$.
We know that this equality holds for each closed point $x\in \varphi(C)$ such that $\varphi^{-1}(x)$ contains a point whose residue field is equal to $k(x)$.

Take any closed point $x\in X$. We will first construct a curve $C'$ passing through $x$ and some finitely many points on $C$ specified below. We will then construct a lisse sheaf on $C'$ whose Frobenius polynomial at $x$ is $f_x(t)$,
and prove that the lisse sheaf on $C'$ extends over $X$ and
 the corresponding representation of $\pi_1(X)$ coincides with $\rho_C$.

We use a lemma by Faltings; 
define $T_0$
to be the set of closed points of $C$ which have the same image in $\Spec\calO_K$ as that of $x$.
By the theorem of Hermite, the Chebotarev density theorem, and the Brauer-Nesbitt theorem, 
there exists a finite set $T\subset \lvert C\rvert\setminus T_0$ satisfying the following properties:
\begin{enumerate}
 \item $T\ra \Spec\calO_K$ is injective.
 \item For any semisimple representations $\rho_1,\rho_2\colon\pi_1(C)\ra \GL_r(E_\lambda)$, the equality $\tr \rho_1(\Frob_y)=\tr \rho_2(\Frob_y)$ for all $y\in T$ implies $\rho_1\cong\rho_2$.
\end{enumerate}
See Satz 5 of \cite{MR718935} or 
Th\'eor\`eme~3.1 of \cite{MR768952} for details.

By Corollary~\ref{cor:Dr-prop2.17} applied to $S=\varphi(T)\cup\{x\}$, there exists a totally real curve $C'$ with a morphism $\varphi'\colon C'\ra X$ such that the map $\varphi'_\ast\colon \pi_1(C')\ra \pi_1(X)/H$ is surjective and for each $y\in \varphi(T)\cup\{x\}$ there exists a point in $\varphi'^{-1}(y)$ whose residue field is equal to $k(y)$.
As discussed before, this pair $(C',\varphi')$ also defines a semisimple representation 
\[
 \rho_{C'}\colon\pi_1(X)\ra\pi_1(X)/H\ra \GL_r(E_\lambda) 
\]
such that $\det\bigl(1-t\rho_{C'}(\Frob_y)\bigr)=f_y(t)$ for each $y\in \varphi(T)\cup\{x\}$. 
Note that the surjectivity of $\varphi_\ast$ implies that
$\rho_{C'}|_{\pi_1(C)}$ is semisimple. 

It follows from property (ii) of $T$ that $\rho_C|_{\pi_1(C)}$ and $\rho_{C'}|_{\pi_1(C)}$ are isomorphic as representations of $\pi_1(C)$. 
Since the map $\varphi_\ast\colon\pi_1(C)\ra \pi_1(X)/H$ is surjective, we have $\rho_C\cong\rho_{C'}$ as representations of $\pi_1(X)/H$ and thus they are also isomorphic as representations of $\pi_1(X)$. 
In particular, we have 
\[
 \det\bigl(1-t\rho_C(\Frob_x)\bigr)=\det\bigl(1-t\rho_{C'}(\Frob_x)\bigr)=f_x(t).
\]
Hence $f$ comes from the lisse sheaf on $X$ corresponding to $\rho_C$.
\end{proof}

We now prove the last proposition of our three key ingredients for Theorem~\ref{thm:Dr-thm2.5}. This concerns extendability of a lisse sheaf on a dense open subset to the whole scheme.
In the proof we use the Zariski-Nagata purity theorem; thus the regularity assumption for $X$ is crucial.
We further need to assume that $X$ is smooth as we use 
Corollary~\ref{cor:Dr-prop2.17} to find a totally real curve passing through 
a given point in a given tangent direction.

\begin{prop}\label{prop:Dr-prop2.14}
Let $K$ be a totally real field and X an irreducible smooth separated $\calO_K[\ell^{-1}]$-scheme with enough totally real curves.
Suppose that $f\in \widetilde{\LS}_r(X)$ satisfies the following conditions:
\begin{enumerate}
  \item $f$ arises from a lisse sheaf over every totally real curve.
  \item There exists a dense open subscheme $U\subset X$ such that $f|_U\in \LS_r(U)$.
\end{enumerate}
Then $f\in \LS_r(X)$.
\end{prop}

\begin{proof}
We follow Section 5.2 of \cite{MR3024821}.
Let $\calE_U$ be the semisimple lisse $E_\lambda$-sheaf on $U$ corresponding to $f|_U$.
First we show that $\calE_U$ extends to a lisse $E_\lambda$-sheaf on $X$.

Suppose the contrary.
Since $X$ is regular, the Zariski-Nagata purity theorem implies
that there exists an irreducible divisor $D$ of $X$ contained in $X\setminus U$
such that $\calE_U$ is ramified along $D$.
Then by a specialization argument (Corollary~5.2 of \cite{MR3024821}), we can find a closed point $x\in X\setminus U$ and a one-dimensional subspace $l\subset T_xX$ satisfying the following property:
\begin{itemize}
\item[($\ast$)]
Consider a triple $(C,c,\varphi)$ consisting of
a regular curve $C$, a closed point $c\in C$, and a morphism $\varphi\colon C\ra X$ such that $\varphi(c)=x$, $\varphi^{-1}(U)\neq\emptyset$, and $\Image \bigl(T_cC\ra  T_xX\otimes_{k(x)}k(c)\bigr)=l\otimes_{k(x)}k(c)$.
For any such triple, the pullback of $\calE_U$ to $\varphi^{-1}(U)$ is ramified at $c$.
\end{itemize}

Let $H$ be the kernel of the representation $\rho_U\colon \pi_1(U)\ra\GL_r(E_\lambda)$ corresponding to $\calE_U$.
The group $\pi_1(U)/H\cong \Image \rho_U$ contains an open pro-$\ell$ subgroup because $\Image\rho_U$ is a compact open subgroup of $\GL_r(E_\lambda)$.
Therefore by Corollary~\ref{cor:Dr-prop2.17} we
find a totally real curve $C$, a closed point $c\in C$, and a morphism $\varphi\colon C\ra X$ such that
\begin{itemize}
 \item $\varphi(c)=x$ and $k(c)\cong k(x)$,
 \item $\varphi^{-1}(U)\neq\emptyset$ and $\varphi_\ast\colon \pi_1\bigl(\varphi^{-1}(U)\bigr)\ra \pi_1(U)/H$ is surjective, and
 \item $\Image(T_cC\ra T_xX)=l$.
\end{itemize}

Since $\varphi_\ast\colon \pi_1\bigl(\varphi^{-1}(U)\bigr)\ra\pi_1(U)/H$ is surjective, the pullback of $\calE_U$ to $\varphi^{-1}(U)$ is semisimple. 
Thus this lisse $E_\lambda$-sheaf has no ramification at $c$ by assumption (i), which contradicts property ($\ast$).
Hence $\calE_U$ extends to a lisse $E_\lambda$-sheaf $\calE$ on $X$.

Let $f'$ be the element of $\LS_r(X)$ corresponding to $\calE$. 
We know $f|_U=f'|_U$.
Take any closed point $x\in X$.
It suffices to show that $f(x)=f'(x)$.
We can find a totally real curve $C'$, a closed point $c'\in C'$, and a morphism $\varphi'\colon C'\ra X$ such that $\varphi'(c')=x$,  $k(c')=k(x)$, and $\varphi'^{-1}(U)\neq \emptyset$.
Then 
\[
 \varphi'^\ast(f)|_{\varphi'^{-1}(U)}=(f|_U)|_{\varphi'^{-1}(U)}
=(f'|_U)|_{\varphi'^{-1}(U)}=\varphi'^\ast(f')|_{\varphi'^{-1}(U)}.
\]
Since $\varphi'^{-1}(U)\neq \emptyset$, the homomorphism $\pi_1\bigl(\varphi'^{-1}(U)\bigr)\ra \pi_1(C')$ is surjective and thus $\varphi'^\ast(f)=\varphi'^\ast(f')$. 
In particular,  $f(x)=\varphi'^\ast(f)(c')=\varphi'^\ast(f')(c')=f'(x)$. 
\end{proof}

\begin{proof}[Proof of Theorem~\ref{thm:Dr-thm2.5}]
First note that a polynomial-valued map $f$ of degree $r$ in the theorem lies in $\widetilde{\LS}_r(X)$.
One direction of the equivalence is obvious, and
thus it suffices to prove that if $f$ satisfies conditions (i) and (ii), then $f$ lies in $\LS_r(X)$.

First assume that $X$ is separated.
Let $k_\lambda$ be the residue field of $E_\lambda$ and 
$N$ be the cardinality of $\GL_r(k_\lambda)$.
Put $X':=X\otimes_{\Z}\Z[N^{-1}]$.
Then $X'\ra X$ is a dominant \'etale morphism and satisfies the following property:
\begin{quote}
 The pullback $f|_{X'}$ arises from a tame lisse sheaf over every separated smooth curve over a finite field.
\end{quote}

Thus by Proposition~\ref{prop:Dr-prop2.12}, there exists an open dense subscheme $U$ of $X$ such that $f|_U$ has a kernel.
Therefore $f|_U$ lies in $\LS_r(U)$ by Proposition~\ref{prop:Dr-prop2.13} and we have $f\in \LS_r(X)$ by Proposition~\ref{prop:Dr-prop2.14}.

In the general case, we consider a covering $X=\bigcup_i U_i$ by open separated subschemes. Then we can apply the above discussion to each $f|_{U_i}$
and obtain a lisse $E_\lambda$-sheaf $\calE_i$ on $U_i$ 
that represents $f|_{U_i}$.
Since $U_i$ is normal, we can replace $\calE_i$ by its semisimplification 
and assume that each $\calE_i$ is semisimple.

Put $U=\bigcap_i U_i$. This is nonempty, and the restrictions
$\calE_i|_U$ are isomorphic to each other.
Thus $\{ \calE_i\}_i$ glues to a lisse $E_\lambda$-sheaf on $X$ and this sheaf represents $f$.
\end{proof}

We end this section with variants of Theorem~\ref{thm:Dr-thm2.5}.
Condition (i) in Theorem~\ref{thm:var of Dr-thm2.5} or Remark~\ref{rem:var of Dr-thm2.5 for unramified case} is weaker than
that of Theorem~\ref{thm:Dr-thm2.5} since they concern only totally real curves with additional properties.
This weaker condition is essential to use the result of \cite{MR3152941} in the proof of our main theorems in the next section. Theorem~\ref{thm:Dr-thm2.5 for CM} is a variant in the CM case.

\begin{thm}\label{thm:var of Dr-thm2.5}
Let $K$ be a totally real field.
Let $X$ be an irreducible smooth $\calO_K[\ell^{-1}]$-scheme with enough totally real curves. 
An element $f\in \widetilde{\LS}_r(X)$ belongs to $\LS_r(X)$ if and only if it satisfies the following conditions:
 \begin{enumerate}
  \item There exists a connected \'etale covering $Y\ra X$ such that $f$ arises from a lisse sheaf over every totally real curve $C$ with the property that $C\times_{X}Y$ is connected.
  \item The restriction of $f$ to each separated smooth curve over a finite field arises from a lisse sheaf.
 \end{enumerate}
\end{thm}

\begin{proof}
Recall that Theorem~\ref{thm:Dr-thm2.5} is deduced from Propositions~\ref{prop:Dr-prop2.12}, \ref{prop:Dr-prop2.13}, and \ref{prop:Dr-prop2.14} and that these propositions have the same condition (i) that $f$ arises from a lisse sheaf over every totally real curve.
Consider the variant statements of Propositions~\ref{prop:Dr-prop2.12}, \ref{prop:Dr-prop2.13}, and \ref{prop:Dr-prop2.14} where 
we replace condition (i) by 
\begin{enumerate}
  \item[(i')] $f$ arises from a lisse sheaf over 
every totally real curve $C$ such that 
$C\times_{X}Y$ is connected.
\end{enumerate}
It suffices to prove that these variants also hold;
then the theorem is deduced from them in the same way as Theorem~\ref{thm:Dr-thm2.5}.

The variant of Proposition~\ref{prop:Dr-prop2.12} is proved in the same way as Proposition~\ref{prop:Dr-prop2.12} if one uses Lemma~\ref{lem:var of totally real elementary fibration} (i) instead of Lemma~\ref{lem:totally real elementary fibration}.
For the variants of Propositions~\ref{prop:Dr-prop2.13} and \ref{prop:Dr-prop2.14}, the same proof works; 
observe that whenever one uses Corollary~\ref{cor:Dr-prop2.17} 
in the proof to find a totally real curve $C$, one can impose the additional condition that $C\times_X Y$ is connected by adding the covering $Y\ra X$ to the input of Corollary~\ref{cor:Dr-prop2.17}. 
\end{proof}

\begin{rem}\label{rem:var of Dr-thm2.5 for unramified case} 
We need another variant of Theorem~\ref{thm:Dr-thm2.5} to prove Theorem~\ref{thm:main thm-special}:
With the notation as in Theorem~\ref{thm:var of Dr-thm2.5},
suppose further that
\begin{itemize}
 \item $K$ is unramified at $\ell$, and
 \item $X$ extends to an irreducible smooth $\calO_K$-scheme $X'$ with nonempty fiber over each place of $K$ above $\ell$.
\end{itemize}

Then condition (i) in Theorem~\ref{thm:var of Dr-thm2.5}
can be replaced by
\begin{enumerate}
  \item[(i')] There exists a connected \'etale covering $Y\ra X$ such that $f$ arises from a lisse sheaf over 
every totally real curve $C$ with the properties that
\begin{itemize}
 \item $C\times_{X}Y$ is connected and that
 \item the fraction field of $C$ is unramified at $\ell$.
\end{itemize}
\end{enumerate}

This statement is proved in the same way as Theorem~\ref{thm:var of Dr-thm2.5}; it suffices to prove variants of Propositions~\ref{prop:Dr-prop2.12}, \ref{prop:Dr-prop2.13}, and \ref{prop:Dr-prop2.14} where condition (i) in these propositions is replaced by the following condition:
\begin{enumerate}
  \item[] The map $f$ arises from a lisse sheaf over 
every totally real curve $C$ such that 
$C\times_{X}Y$ is connected and the fraction field of $C$
is unramified at $\ell$.
\end{enumerate}

For the proof of the variant of Proposition~\ref{prop:Dr-prop2.12},
we also need to consider the variant of Lemma~\ref{lem:var of totally real elementary fibration} (i) where we further require that the fraction field of $X_1$ is unramified at $\ell$.

We now explain how to prove the variants of Lemma~\ref{lem:var of totally real elementary fibration} (i) and Propositions~\ref{prop:Dr-prop2.12}, \ref{prop:Dr-prop2.13}, and \ref{prop:Dr-prop2.14}.
By the additional condition on $X$, for each place $v$ of $K$ above $\ell$,
 there exist a finite unramified extension $L$ of $K_v$
and a morphism $\Spec \calO_L\ra X'$.
We denote the image of the closed point of $\Spec\calO_L$ by $s_v$.
Since $L$ is unramified over $K_v$, we can find
a horizontal one-dimensional subspace $l_v$ of $T_{s_v}X'$
with respect to $X'\ra \Spec\calO_K$.

If we add $\{s_v\}_{v|\ell}$ and $l_v$ to the input when we use Corollary~\ref{cor:Dr-prop2.17}, the fraction field of the resulting totally real curve is unramified over $K$ at each $v$, hence unramified at $\ell$.
Thus we can prove the variant of Lemma~\ref{lem:var of totally real elementary fibration} (i) in the same way as Lemma~\ref{lem:var of totally real elementary fibration} (i), and the arguments given in Theorem~\ref{thm:var of Dr-thm2.5} work for the current variants of Propositions~\ref{prop:Dr-prop2.12}, \ref{prop:Dr-prop2.13}, and \ref{prop:Dr-prop2.14}.
Hence the statement of this remark follows. 
\end{rem}

\begin{thm}
\label{thm:Dr-thm2.5 for CM}
Let $F$ be a CM field.
Let $Z$ be an irreducible smooth $\calO_F[\ell^{-1}]$-scheme with geometrically irreducible generic fiber.
An element $f\in \widetilde{\LS}_r(Z)$ belongs to $\LS_r(Z)$ if and only if it satisfies the following conditions:
 \begin{enumerate}
  \item There exists a connected \'etale covering $Y\ra Z$ such that $f$ arises from a lisse sheaf over every CM curve $C$ with the property that $C\times_{Z}Y$ is connected. 
  \item The restriction of $f$ to each separated smooth curve over a finite field arises from a lisse sheaf.
 \end{enumerate}
\end{thm}

\begin{proof}
We can prove variants of Propositions~\ref{prop:Dr-prop2.12}, \ref{prop:Dr-prop2.13}, and \ref{prop:Dr-prop2.14} for the CM case using Theorem~\ref{thm:Dr-thm2.15 for CM} and Lemma~\ref{lem:var of totally real elementary fibration} (ii). 
Then the theorem is deduced from them in the same way as Theorems~\ref{thm:Dr-thm2.5} and \ref{thm:var of Dr-thm2.5}.
\end{proof}

\section{Proofs of the main theorems}
\label{section:prf of main thm}
In this section, we prove theorems on the existence of the compatible system of a lisse sheaf.
Theorem~\ref{thm:main thm} concerns the totally real case and
Theorem~\ref{thm:main thm for CM} concerns the CM case.
Theorem~\ref{thm:main thm-special} in the introduction is proved after Theorem~\ref{thm:main thm}.
Following the discussion in the Subsection~2.3 of \cite{MR3024821}, we deduce these main theorems from Theorems~\ref{thm:var of Dr-thm2.5}, \ref{thm:Dr-thm2.5 for CM}, and theorems in \cite{MR1875184} and \cite{MR3152941}.

As we mentioned in the introduction, some of the assumptions in the main theorems come from the potential diagonalizability condition, which is introduced in \cite{MR3152941}. We first review this notion. See Section~1.4 of \cite{MR3152941} for details.

Let $L$ be a finite extension of $\Q_\ell$. Let $E_\lambda$ be a finite extension of $\Q_\ell$.
We say that an $\overline{E}_\lambda$-representation $\rho$ of $\Gal(\overline{L}/L)$ is \emph{potentially diagonalizable} if it is potentially crystalline and there is a finite extension $L'$ of $L$ such that $\rho|_{\Gal(\overline{L}/L')}$ lies on the same irreducible component of the universal crystalline lifting ring of the residual representation $\overline{\rho}|_{\Gal(\overline{L}/L')}$ with fixed Hodge-Tate numbers as a sum of characters lifting $\overline{\rho}|_{\Gal(\overline{L}/L')}$.

There are two important examples of this notion (see Lemma~1.4.3 of \cite{MR3152941}): 
Ordinary representations are potentially diagonalizable.
When $L$ is unramified over $\Q_\ell$, a crystalline representation
is potentially diagonalizable if for each $\tau\colon L\hra \overline{E}_\lambda$
the $\tau$-Hodge-Tate numbers lie in the range $[a_{\tau},a_{\tau}+\ell-2]$ for some integer $a_{\tau}$.

We first prove our main theorem for the totally real case.

\begin{thm}\label{thm:main thm}
Let $\ell$ be a rational prime.
Let $K$ be a totally real field and $X$ an irreducible smooth $\calO_K[\ell^{-1}]$-scheme with enough totally real curves. 
Let $E$ be a finite extension of $\Q$ and $\lambda$ a prime of $E$ above $\ell$.
Let $\calE$ be a lisse $E_\lambda$-sheaf on $X$ and $\rho$ the corresponding representation of $\pi_1(X)$.
Suppose that $\calE$ satisfies the following assumptions:
\begin{enumerate}
 \item The polynomial $\det (1-\Frob_xt,\calE_{\bar{x}})$ has coefficients in $E$ for every $x\in \lvert X\rvert$.
 \item For every totally real field $L$ and every morphism $\alpha\colon \Spec L\ra X$, 
the $E_\lambda$-representation $\alpha^\ast\rho$ of $\Gal(\overline{L}/L)$ is potentially diagonalizable at each prime $v$ of $L$ above $\ell$ and for each
$\tau\colon L\hra \overline{E}_\lambda$ it has 
distinct $\tau$-Hodge-Tate numbers.
 \item $\rho$ can be equipped with symplectic (resp.~ orthogonal) structure with multiplier $\mu\colon \pi_1(X)\ra E_\lambda^\times$ 
such that 
$\mu|_{\pi_1(X_K)}$ admits a factorization
\[
\mu|_{\pi_1(X_K)}\colon \pi_1(X_K)\lra\Gal(\overline{K}/K)\stackrel{\mu_K}{\lra} E_\lambda^\times
\]
with a totally odd (resp.~ totally even) character $\mu_K$.
 \item The residual representation $\overline{\rho}|_{\pi_1(X[\zeta_\ell])}$ is absolutely irreducible.
 \item $\ell\geq 2(\rank \calE+1)$.
\end{enumerate}
Then for each rational prime $\ell'$ and each prime $\lambda'$ of $E$ above $\ell'$ there exists a lisse $\overline{E}_{\lambda'}$-sheaf on $X[\ell'^{-1}]$ which is compatible with $\calE|_{X[\ell'^{-1}]}$.
\end{thm}

\begin{proof}
Replacing $X$ by $X[\ell'^{-1}]$, we may assume that $\ell'$ is invertible in $\calO_X$.
Let $r$ be the rank of $\calE$.
Take an arbitrary extension $M$ of $E_{\lambda'}$ of degree $r!$. 
By assumption (i), we regard the map $f\colon x\mapsto \det(1-\Frob_xt,\calE_{\bar{x}})$ as an element of $\widetilde{\LS}^{M}_r(X)$ via the embedding $E\hra E_{\lambda'}\hra M$.

We will apply Theorem~\ref{thm:var of Dr-thm2.5} to $f\in \widetilde{\LS}^{M}_r(X)$.
Here we use the prime $\ell'$ and the field $M$ (we used $\ell$ and $E_\lambda$ in Section~\ref{section:prf of Dr-thm2.5}).

First we show that the map $f$ satisfies condition (i) in Theorem~\ref{thm:var of Dr-thm2.5}.
Let $Y$ be the connected \'etale covering $Y\ra X[\zeta_{\ell}]$
that corresponds to $\Ker \overline{\rho}|_{\pi_1(X[\zeta_\ell])}$.
We regard $Y$ as a connected \'etale covering over $X$
via $Y\ra X[\zeta_{\ell}]\ra X$.
We will prove that this $Y\ra X$ satisfies condition (i).
Take any totally real curve $\varphi\colon C\ra X$ 
such that $C\times_{X}Y$ is connected.

To show that $\varphi^\ast (f)$ arises from a lisse $M$-sheaf on $C$,
it suffices to prove that there exists a lisse $\overline{E}_{\lambda'}$-sheaf on $C$ which is compatible with $\varphi^\ast\calE$;
this follows from Lemma~2.7 of \cite{MR3024821}. 
Namely, let $\rho_C'\colon \pi_1(C)\ra \GL_r(\overline{E}_{\lambda'})$ denote the semisimplification of the corresponding $\overline{E}_{\lambda'}$-representation. 
Since $\det\bigl(1-t\rho_C'(\Frob_x)\bigr)=f_x(t)\in E_{\lambda'}[t]$ for every closed point $x\in C$, the character of $\rho_C'$ is defined over $E_{\lambda'}$ by the Chebotarev density theorem.
It follows from $[M: E_{\lambda'}]=r!$ that the Brauer obstruction of $\rho_C'$ in the Brauer group $\Br(E_{\lambda'})$ vanishes in $\Br(M)$ and
$\rho_C'$ can be defined over $M$.
This means $\varphi^\ast(f)\in \LS_r^M(C)$.

We will construct a lisse $\overline{E}_{\lambda'}$-sheaf on $C$ which is compatible with $\varphi^\ast\calE$.
For this we will apply
Theorem C of \cite{MR3152941} to the $\overline{E}_\lambda$-representation 
$\varphi_L^\ast \rho$ of $\Gal(\overline{L}/L)$,
where $L$ denotes the fraction field of $C$ and
$\varphi_L\colon \Spec L\ra X$ denotes $\varphi|_{\Spec L}$.

We need to see that the Galois representation $\varphi_L^\ast \rho$ satisfies the assumptions in Theorem C.
By assumptions (ii) and (v) it remains to check that
\begin{enumerate}
 \item[(a)] $\varphi_L^\ast \rho$ can be equipped with symplectic (resp.~ orthogonal) structure with totally odd (resp.~ totally even) multiplier, and
 \item[(b)] the residual representation $(\varphi_L^\ast \overline{\rho})|_{\Gal(\overline{L}/L(\zeta_\ell))}$ is absolutely irreducible.
\end{enumerate}

Assumption (a) follows from assumption (iii).
To see (b), recall that $C\times_XY$ is connected.
Hence $C\times_XX[\zeta_{\ell}]$ is connected with fraction field $L(\zeta_\ell)$, and $C\times_XY\ra C\times_XX[\zeta_{\ell}]$ is a connected \'etale covering. It follows from the definition of $Y$ that
$\Image (\varphi_L^\ast \overline{\rho})|_{\Gal(\overline{L}/L(\zeta_\ell))}$
coincides with
$\Image\overline{\rho}|_{\pi_1(X[\zeta_\ell])}$, and thus 
$(\varphi_L^\ast \overline{\rho})|_{\Gal(\overline{L}/L(\zeta_\ell))}$ is absolutely irreducible by assumption (iv).

Hence by Theorem C of \cite{MR3152941} we obtain an $\overline{E}_{\lambda'}$-representation of $\Gal(\overline{L}/L)$. The proof of the theorem, which uses potential automorphy and Brauer's theorem, shows that this representation is unramified at each closed point of $C$, and thus
it gives rise to a lisse $\overline{E}_{\lambda'}$-sheaf on $C$
which is compatible with $\varphi^\ast\calE$.
Hence $f$ satisfies condition (i) in Theorem~\ref{thm:var of Dr-thm2.5}.

Next we show that $f$ satisfies condition (ii) in Theorem~\ref{thm:var of Dr-thm2.5}.
Let $C$ be a separated smooth curve
over $\F_p$ for some prime $p$ and 
denote the structure morphism $C\ra \Spec \F_p$ by $\alpha$.
Let $\varphi\colon C\ra X$ be a morphism.
Note that $p$ is different from $\ell$ and $\ell'$.

We write the semisimplification of $\varphi^\ast\calE$ as $\bigoplus_i\calE_i^{\oplus r_i}$, where $\calE_i$ are distinct irreducible lisse $\overline{E}_\lambda$-sheaves on $C$.
Then there exist an irreducible lisse $\overline{E}_\lambda$-sheaf $\calF_i$ on $C$ and a lisse $\overline{E}_\lambda$-sheaf $\calG_i$ of rank 1 on $\Spec \F_p$ such that $\calF_i$ has determinant of finite order and $\calE_i\cong \calF_i\otimes \alpha^\ast\calG_i$ (see Section I.3 of \cite{MR601520} or Section 0.4 of \cite{MR3024820} 
for example).

By Th\'{e}or\`{e}m VII.6 of \cite{MR1875184}, for each closed point $x\in C$, the roots of $\det(1-\Frob_xt,\calF_{i,\bar{x}})$ are algebraic numbers that are $\lambda'$-adic units. Moreover, there exists an irreducible lisse $\overline{E}_{\lambda'}$-sheaf $\calF_i'$ on $C$ which is compatible with $\calF_i$.

We will show that there exists a lisse $\overline{E}_{\lambda'}$-sheaf $\calG'_i$ on $\Spec \F_p$ which is compatible with $\calG_i$.
Note that the lisse $\overline{E}_\lambda$-sheaf $\calG_i$ is determined by the value of the corresponding character of $\Gal(\overline{\F}_p/\F_p)$
at the geometric Frobenius. Denote this value by $\beta_i \in \overline{E}_\lambda^\times$. 
 It suffices to prove that $\beta_i$ is an algebraic number that is a $\lambda'$-adic unit.
Since the roots of $\det(1-\Frob_xt,\calE_{\bar{x}})$ and $\det(1-\Frob_xt,\calF_{i,\bar{x}})$ are all algebraic numbers,
so is $\beta_i$.

We prove that $\beta_i$ is a $\lambda'$-adic unit. 
To see this, take a closed point $x$ of $C$.
Then by Corollary~\ref{cor:Dr-prop2.17}
we can find a totally real curve $C'$ and a morphism $\varphi'\colon C'\ra X$ such that $\varphi(x)\in \varphi'(C')$ and  $C'\times_XY$ is connected.
As discussed before, Theorem C of \cite{MR3152941} implies that
there exists a lisse $\overline{E}_{\lambda'}$-sheaf on $C'$ whose Frobenius characteristic polynomial map is $\varphi'^\ast(f)$. 
Thus for each closed point $y\in C'$ the roots of $\varphi'^\ast(f)(y)$
are algebraic numbers that are $\lambda'$-adic units.
Considering a point $y\in \varphi'^{-1}(\varphi(x))$,
we conclude that some power of $\beta_i$ is a $\lambda'$-adic unit and thus so is $\beta_i$.
Hence there exists a lisse $\overline{E}_{\lambda'}$-sheaf $\calG'_i$ on $\Spec \F_p$ which is compatible with $\calG_i$.

The Frobenius characteristic polynomial map associated with the semisimple lisse $\overline{E}_{\lambda'}$-sheaf $\bigoplus_i(\calF'_i\otimes\alpha^\ast\calG'_i)^{\oplus r_i}$ is $\varphi^\ast(f)$.
As discussed before, this sheaf can be defined over $M$. 
Thus $f$ satisfies condition (ii) in Theorem~\ref{thm:var of Dr-thm2.5}.

Therefore by Theorem~\ref{thm:var of Dr-thm2.5}
there exists a lisse $M$-sheaf on $X$ which is compatible with $\calE$.
\end{proof}

\begin{proof}[Proof of Theorem~\ref{thm:main thm-special}]
All the discussions in the proof of Theorem~\ref{thm:main thm}
also work in this setting by using Remark~\ref{rem:var of Dr-thm2.5 for unramified case} instead of Theorem~\ref{thm:var of Dr-thm2.5}.
\end{proof}

We also have a theorem for the CM case.

\begin{thm}\label{thm:main thm for CM}
Let $\ell$ be a rational prime, $E$ a finite extension of $\Q$, and $\lambda$ a prime of $E$ above $\ell$.
Let $F$ be a CM field with $\zeta_\ell\not\in F$ and $Z$ an irreducible smooth $\calO_F[\ell^{-1}]$-scheme with geometrically irreducible generic fiber.
Let $\calE$ be a lisse $E_\lambda$-sheaf on $X$ and $\rho$ the corresponding representation of $\pi_1(Z)$.
Suppose that $\calE$ satisfies the following assumptions:
\begin{enumerate}
 \item The polynomial $\det (1-\Frob_xt,\calE_{\bar{x}})$ has coefficients in $E$ for every $x\in \lvert X\rvert$.
 \item For any CM field $L$ with $\zeta_\ell\not\in L$ and any morphism $\alpha\colon \Spec L\ra Z$, the $E_\lambda$-representation $\alpha^\ast\rho$ of $\Gal(\overline{L}/L)$ satisfies the following two conditions:
\begin{enumerate}
 \item $\alpha^\ast\rho$ is potentially diagonalizable 
       at each prime $v$ of $L$ above $\ell$ and 
       for each $\tau\colon L \hra \overline{E}_\lambda$
       it has distinct $\tau$-Hodge-Tate numbers.
 \item $\alpha^\ast\rho$ is totally odd and polarizable in the sense of 
Section 2.1 of \cite{MR3152941}.
\end{enumerate}
 \item The residual representation $\overline{\rho}|_{\pi_1(Z[\zeta_\ell])}$ is absolutely irreducible.
 \item $\ell\geq 2(\rank \calE+1)$.
\end{enumerate}
Then for each rational prime $\ell'$ and each prime $\lambda'$ of $E$ above $\ell'$ there exists a lisse $\overline{E}_{\lambda'}$-sheaf on $Z[\ell'^{-1}]$ which is compatible with $\calE|_{Z[\ell'^{-1}]}$.
\end{thm}

\begin{proof}
In the same way as Theorem~\ref{thm:main thm}, the theorem is deduced from Theorem~\ref{thm:Dr-thm2.5 for CM}, Theorem~5.5.1 of \cite{MR3152941}, Th\'{e}or\`{e}m VII.6 of \cite{MR1875184}, and the following remark:
If $Y\ra Z[\zeta_\ell]$ denotes the connected \'etale covering defined by $\Ker \overline{\rho}_{\pi_1(Z[\zeta_\ell])}$ and $C$ is a CM curve with a morphism to $Z$ such that $C\times_ZY$ is connected, then $C\times_ZZ[\zeta_{\ell}]=C\otimes_{\calO_F}\calO_F(\zeta_\ell)$ is connected. In particular, the fraction field of $C$ does not contain $\zeta_{\ell}$.
\end{proof}

\end{document}